\documentclass[sn-mathphys-num]{sn-jnl}

\usepackage{graphicx}%
\usepackage{multirow}%
\usepackage{amsmath,amssymb,amsfonts}%
\usepackage{amsthm}%
\usepackage{mathrsfs}%
\usepackage[title]{appendix}%
\usepackage{xcolor}%
\usepackage{textcomp}%
\usepackage{manyfoot}%
\usepackage{booktabs}%
\usepackage{algorithm}%
\usepackage{algorithmicx}%
\usepackage{algpseudocode}%
\usepackage{listings}%
\usepackage[justification = centering,labelsep = period]{caption}
\usepackage{geometry}
\usepackage{arydshln}
\usepackage{adjustbox} 
\usepackage{tikz}
\usetikzlibrary{arrows.meta, matrix, positioning,calc}

\theoremstyle{thmstyleone}%
\newtheorem{theorem}{Theorem}[section]
\newtheorem{lem}{Lemma}[section]
\theoremstyle{thmstyletwo}%
\newtheorem{example}{Example}%
\newtheorem{remark}{Remark}%
\newtheorem{cor}{Corollary}[section]
\theoremstyle{thmstylethree}%
\newtheorem{definition}{Definition}[section]%

\begin{document}

\title[Article Title]{ The  Schur complements for $SDD_{1}$ matrices and their application to linear complementarity problems    }


\author[1]{\fnm{Yang} \sur{Hu}}

\author*[1,2]{\fnm{Jianzhou} \sur{Liu}}\email{liujz@xtu.edu.cn}

\author[3]{\fnm{Wenlong} \sur{Zeng}}

\affil[1]{ School of Mathematics and Computational Science, Xiangtan University, Xiangtan, 411105,
	Hunan, People’s Republic of China}

\affil[2]{ Hunan Key Laboratory for Computation and Simulation in Science and Engineering, Xiangtan
	University, Xiangtan, People’s Republic of China}
	
	\affil[3]{Department of Mathematics, Shanghai University, Shanghai 200444,  People’s Republic of China}


\abstract{ In this paper we propose a new scaling method to study the Schur complements of $SDD_{1}$  matrices. Its core is related to the non-negative property of the inverse $M$-matrix, while numerically improving the Quotient formula.  Based on the Schur complement and a novel norm splitting manner, we establish an upper bound for the infinity norm of the inverse  of $SDD_{1}$ matrices, which depends solely on the original matrix entries.  We apply the new bound to derive an error bound for linear complementarity problems of $B_{1}$-matrices. Additionally, new lower and upper bounds for the determinant of $SDD_{1}$ matrices are presented.       Numerical experiments validate the effectiveness and superiority of our results.}



\maketitle
\section{Introduction}\label{sec1}

The linear complementarity problem, denoted by  LCP($M,q$), consists of finding a vector  $x\in \mathbb{R}^n $ satisfying
\begin{align*}
	x\geq 0,Mx+q\geq 0,(Mx+q)^{T}x=0,
\end{align*}
or proving the nonexistence of such a solution. Here, $M=(m_{ij})\in{\mathbb R^{n\times n}}$ and  $ q\in  \mathbb{R}^n$. The LCP($M,q$) has various applications,  such as the problems of finding a Nash equilibrium point of bimatrix game, linear and quadratic programming or some free boundary problems of fluid mechanics and optimal stopping in Markov chains \cite{Berman A, Schafer,Cottle}.

\par 
The LCP($M,q$) has unique solution for any  $ q\in  \mathbb{R}^n$ if and only if $M$ is a $P$-matrix \cite{Cottle}. A matrix $M\in{\mathbb R^{n\times n}}$
is called a $P$-matrix if all its principal minors are positive. Furthermore,  $x^{*}$ solves the LCP($M,q$) if and only if  $x^{*}$ solves
\begin{align*}
	r(x):={\rm min}\{x,Mx+q\}=0,
\end{align*}
where the min operator $r(x)$ denotes the componentwise minimum of two vectors. When the involved matrix is a $P$-matrix, Chen and Xiang gave the following error bound for the LCP($M,q$) \cite{Chen X}:
\begin{align*}
	\left\|x-x^{*}\right\|_{\infty}\leq \underset{\rm d\in[0,1]^{n}}{\rm max}\left\|(I-D+DM)^{-1}\right\|_{\infty} \left\|r(x)\right\|_{\infty} ,
\end{align*}
where $x^{*}$ is the solution of the LCP($M,q$), $D=diag\{d_{1},d_{2},\dots,d_{n}\}$ over  $0\leq d_{i}\leq 1$ for  $i\in N$. Therefore, one of the important problems regarding the LCP($M,q$) is to estimate :
\begin{align}
	\underset{\rm d\in[0,1]^{n}}{\rm{max}}\left\|(I-D+DM)^{-1}\right\|_{\infty}. \label{0.1}
\end{align}\label{new 1}
Since it plays a significant role in gauging the error bound.  This involves finding an upper bound for $	\left\|M^{-1}\right\|_{\infty}$. 
Since any principal submatrix of a $P$-matrix is  nonsingular, based on this property, we can partition matrix $M$ into
\begin{align*}
	M=\left(
	\begin{array}{cc}
		A&B\\
		C&D\\
	\end{array}
	\right),
\end{align*}
where $A=(a_{ij})\in \mathbb R^{m\times m}(0<m<n)$. Denote $M/A=D-CA^{-1}B$,  we have
\begin{align*}
	XMY=\left(
	\begin{array}{cc}
		A&O  \\
		O &M/A  \\
	\end{array}
	\right),
\end{align*}
where 
\begin{align*}
	X=\left(
	\begin{array}{cc}
		I_{m}&O  \\
		-	CA^{-1} &I_{n-m}   \\
	\end{array}
	\right) \quad {\text{and}} \quad 	Y=\left(
	\begin{array}{cc}
		I_{m}&-A^{-1}B  \\
		O &I_{n-m}   \\
	\end{array}
	\right).
\end{align*}
Furthermore, we obtain that
\begin{align*}
	M^{-1}=Y\left(
	\begin{array}{cc}
		A^{-1}&O  \\
		O &(M/A)^{-1}   \\
	\end{array}
	\right)X.
\end{align*}
By the fundamental properties of the infinity norm, we have:
\begin{align}
	\left\|M^{-1}\right\|_{\infty}\leq  \left\|Y\right\|_{\infty} {\rm max} \{  	\left\|A^{-1}\right\|_{\infty},  	\left\|(M/A)^{-1}\right\|_{\infty}    \}	\left\|X\right\|_{\infty}.  \label{new2}
\end{align}
 This demonstrates that  (\ref{new2}) transforms the upper bound estimation of 	$\left\|M^{-1}\right\|_{\infty}$ into four small-scale matrix upper bounds.  Building on this framework, scholars have obtained upper bounds of $\left\|M^{-1}\right\|_{\infty}$ for certain subclasses of $P$-matrices \cite{Li C,  Sang, ZWL}. However,  the estimation provided by (\ref{new2}) exhibits significant inaccuracy, particularly when $	\left\|A^{-1}\right\|_{\infty}$ and $\left\|X\right\|_{\infty}$ are large. In this paper, we establish a strictly sharper upper bound for $\left\|M^{-1}\right\|_{\infty}$ that demonstrably outperforms (\ref{new2}):
\begin{align}
	\left\|M^{-1}\right\|_{\infty}\leq  \left\|Y\right\|_{\infty} {\rm max} \{  	\left\|A^{-1}\right\|_{\infty},  	\left\|(M/A)^{-1}\right\|_{\infty} \left\|X\right\|_{\infty}   \}. 
\end{align}\label{new 3}
\quad Now, let us formally define the Schur complement. Let $n$ be apositive integer, where $n\geq 2$, and define the set $N=\{1,2,\dots,n\}$.     For nonempty index sets $\alpha, \beta\subsetneq N$, we denote by $|\alpha|$ the cardinality  of $\alpha$ and $\overline \alpha=N\backslash \alpha$ the complement of $\alpha$ in  $N$. We write $A(\alpha, \beta)$ to mean the submatrix of $A\in \mathbb C^{n\times n}$ lying in the rows indexed by $\alpha$ and the columns indexed by $\beta$. $A(\alpha, \alpha)$ is abbreviated to $A(\alpha)$.   Assume that $A(\alpha)$ is nonsingular, the Schur complement of $A$ with respect to $A(\alpha)$, which is denoted by $A/A(\alpha)$ or simply $A/\alpha$, is defined as
\begin{equation}
	A/\alpha=A(\overline\alpha)-A(\overline\alpha,\alpha)A(\alpha)^{-1}A(\alpha,\overline\alpha).
\end{equation}
The  Schur complement is a useful tool in many fields such as control theory, numerical algebra, polynomial optimization, magnetic resonance imaging and simulation \cite{Liu SIAM, Horn, Torun,Zqq,Zfz}.   When employing Schur complement techniques to reduce matrix dimensionality, one of the important problems is whether some important properties or structures of the given matrices are inherited by the Schur complements.  It is known that the Schur complements of positive semidefinite matrices are positive semidefinite, the same is true of $M$-matrix, and $H$-matrix.

  In  2010, Pe\~{n}a introduced a novel subclass of $H$-matrices known as $SDD_{1}$ matrices \cite{Pena}.  It has significant applications in many fields such as $H$-matrix identification, determination of strong $H$-tensors, mathematical programming \cite{Dai,Li Q,Kolot, Sk, WYQ}.   Furthermore, the derived  $B_{1}$-matrix  from $SDD_1$  matrix constitutes a significant subclass of $P$-matrices. To date, no precise upper bound related to (\ref{new 1}) has been established for $B_{1}$-matrices, which motivates our work to address this fundamental gap.  Additionally, 
  Pe\~{n}a\cite{Pena} provided an example demonstrating that the Schur complement of an $SDD_{1}$ matrix is not necessarily $SDD_{1}$. This naturally raises an important question: Under what conditions does the Schur complement of an $SDD_{1}$ matrix preserve its $SDD_{1}$ property? 

In this paper, we propose a new scaling method called prior construction  and apply it to the study of the Schur complement for $SDD_{1}$ matrices.  The core of this method lies in the non-negative property of the inverse matrix of $M$-matrix. Using the prior construction method, our main contributions are as follows:
\begin{itemize}
	\item[\textbullet] When $\emptyset\neq \alpha\subsetneq N_{2}$, $A/\alpha$ remains an $SDD_{1}$ matrix.  More specifically, under this condition, the $D_{1}$ diagonal dominant degree of $A/\alpha$ is stronger than that of the original matrix in the corresponding rows, i.e.,
	\begin{align*}
		|a^{\prime}_{tt}|-P_{t}(A/\alpha)\geq |a_{j_{t}j_{t}}|-P_{j_{t}}(A)>0.
	\end{align*}
\end{itemize}

\begin{itemize}
	\item[\textbullet] When $\alpha=N_{2}$,  $A/\alpha$ is an SDD matrix. Based on the Schur complements, we give a new upper bound for the infinity norm of the inverse of an $SDD_{1}$ matrix. Building on this bound, we establish an error bound for the linear complementarity problem of $B_{1}$-matrices, which only  dependent on the entries of the original matrix. 
\end{itemize}
\begin{align*}
	\left\|A^{-1}\right\|_{\infty} \leq  \left(1+\underset{i\in N_{2}}{\rm max}\frac{P_{i}(A)}{|a_{ii}|}\right){\rm max}\{\phi,\psi\}.
\end{align*}

\begin{itemize}
	\item[\textbullet] When $N_{2}\subsetneq \alpha \subsetneq N$, we give the exact lower bounds of $|a^{\prime}_{tt}|-R_{t}(A/\alpha)$, which effectively addresses the question that the Quotient formula cannot answer. As an application, we determine an upper and a lower bounds for the determinant of the $SDD_{1}$ matrix. Notably, it is found that our result outperforms Huang' upper and lower bounds \cite{Htz} in a specific ordering.
	\begin{align*}
		\prod\limits_{i=1}^n\left(|a_{ii}|-\sum_{j=i+1}^{n}|a_{ij}|y_{j}\right)\leq |\det(A)|\leq 	\prod\limits_{i=1}^n\left(|a_{ii}|+\sum_{j=i+1}^{n}|a_{ij}|y_{j}\right).
	\end{align*}
\end{itemize}

The rest of the paper is organized as follows. In section \ref{section.2},  we present the necessary definitions and lemmas. In section \ref{section.3}, we propose the prior construction method and analyze the Schur complements of $SDD_{1}$ matrices under three distinct cases. In Section \ref{section.4}, based on the Schur complement, we establish a new upper bound for the infinity norm of the inverse of $SDD_{1}$ matrices. Moreover, we propose a new subclass of $H$-matrices, termed $S$-$SDD_{1}$ matrices. In Section \ref{section.5}, we derive error bound for the linear complementarity problem of $B_{1}$-matrices.   In Section \ref{section.6}, we derive new upper and lower bounds for the determinant of $SDD_{1}$ matrices. Numerical experiments validate the effectiveness and superiority of our proposed results.

\section{Preliminaries}\label{section.2}
To establish our principal results, we first introduce requisite definitions and preliminary lemmas. Given two matrices  $A=(a_{ij})\in \mathbb C^{n\times n}$ and $B=(b_{ij})\in  \mathbb C^{n\times n}$, $A\geq B$ means that $a_{ij}\geq b_{ij}$ for all $i,j\in N$. The modulus matrix is defined as $|A|=(|a_{ij}|)$.  Denote
\begin{align*}
	R_{i}(A)=\underset{j\in{N\backslash{\{i\}}}}{\sum}|a_{ij}|,
\end{align*}
and
\begin{align*}
 N_{1}=\{i\in{N}:|a_{ii}|\leq R_{i}(A)\}, \quad 	N_{2}=\{i\in{N}:|a_{ii}|>R_{i}(A)\}.
\end{align*}
We call $N_{1}$ as the non-diagonally dominant set of $A$ and $N_{2}$ as the diagonally dominant set of $A$.

\begin{definition} \rm
	\cite{Varga}  Let $A=(a_{ij})\in{\mathbb C^{n\times n}}$. If, for all $i\in N$
	\begin{align*}
		|a_{ii}|>R_{i}(A),
	\end{align*}
then $A$ is called a strictly diagonally dominant (SDD) matrix.\\
\end{definition}

\begin{definition}\rm \cite{Pena}  Let $A=(a_{ij})\in \mathbb C^{n\times n}$$(n\geq 2)$. If, for all $i\in N$ 
	\begin{align}
		|a_{ii}|>P_{i}(A):=\underset{j\in N_{1}\backslash \{i\}}\sum|a_{ij}|+\underset{j\in N_{2}\backslash \{i\}}\sum|a_{ij}|\frac{R_{j}(A)}{|a_{jj}|}, \label{eq 2.1}
	\end{align}
then $A$ is called an $SDD_{1}$ matrix.
\end{definition}

\begin{definition}\rm \cite{Berman A}
	A matrix $A$ is called an  $M$-matrix if there exist a nonnegative matrix $B$ and a real number $s>\rho(B)$ such that $A=sI-B$, where $\rho(B)$ is the spectral radius of $B$.
\end{definition}

\begin{definition}\rm \cite{Horn} Given a matrix $A=(a_{ij})\in \mathbb C^{n\times n}$, the comparison matrix of $A$ denoted by $< A>=(\mu_{ij})$, is defined as
	\begin{align*}
		\mu_{ij}=&\left\{\begin{aligned}\notag
			|a_{ii}|,  i=j, \\
			\thinspace
			-|a_{ij}|,i\neq{j}.
		\end{aligned}\right. 
	\end{align*}
\end{definition}

\begin{definition}\rm \cite{Horn}
	A matrix  $A$  is called an $H$-matrix if $< A>$ is an $M$-matrix.
\end{definition}

\begin{lem}\label{Lemma 2.1}\rm \cite{Varga}
	If $A$ is  an  $H$-matrix, then
	\begin{align*}
		< A>^{-1}\geq |A^{-1}|.
	\end{align*}
\end{lem}

\begin{lem} \rm \cite{Pena} The matrix classes $SDD$, $SDD_{1}$, $H$ satisfy the following inclusion relations:
	\begin{align*}
		\{SDD\}\subsetneq \{SDD_{1} \}\subsetneq \{H  \}.
	\end{align*}
	\end{lem}

\begin{lem}\label{Lemma 2.2}\rm
	Let $A=(a_{ij})\in \mathbb C^{n\times n}$, $x,y\in \mathbb C^{n}$. If $A\geq 0$ and $x\leq y$, then $Ax\leq Ay$.
\end{lem}
\begin{proof}
	Denote $x=(x_{1},x_{2},\dots,x_{n})^{T}$, $y=(y_{1},y_{2},\dots,y_{n})^{T}$. Obviously $x_{i} \leq y_{i}$,  for all $i\in N$. Let
	\begin{align*}
		z = A(y - x)=(z_{1},z_{2},\dots,z_{n})^{T}.
	\end{align*}
	Then, for each $i\in N$,
	\begin{align*}
		z_{i} = \sum_{j=1}^{n} a_{ij} (y_{j} - x_{j}) \geq 0.
	\end{align*}
	Thus, we have $A(y - x) \geq 0$, which implies $Ay \geq Ax$. The proof is completed.
\end{proof}
\begin{lem}\rm  \label{Lemma 2.3} \cite{Zfz} 
	(Quotient formula) Let $A$ be a square matrix. If $B$ is a proper nonsingular principal submatrix of $A$ and $C$ is a proper nonsingular principal submatrix of $B$, then $B/C$ is a proper nonsingular principal submatrix of $A/C$, and
	\begin{align*}
		A/B=(A/C)/(B/C).
	\end{align*}
\end{lem}

\begin{lem}\rm \label{Lemma 2.4} \cite{Zfz}
	(Schur’s   formula) Let $A=(a_{ij})\in \mathbb C^{n\times n}$, if
	\begin{align*}
		A=\left(
	\begin{array}{cc}
		A(\alpha)&A(\alpha,\overline \alpha)\\
		A(\overline \alpha,\alpha)&A(\overline \alpha)\\
	\end{array}
	\right),
	\end{align*}
where $A(\alpha)$ is a nonsingular leading principal submatrix of $A$, then 
\begin{align*}
	det(A)=det\left(A(\alpha)\right)\times det(A/\alpha).
\end{align*}
\end{lem}

\begin{lem}\rm \label{Lemma 2.5} If $a>b>0$ and $c\geq 0$, then $\frac{b+c}{a+c}\geq \frac{b}{a}$.
	\end{lem}

\begin{lem}\label{Lemma 2.6}\rm
	\cite{ZWL} Let $\gamma>0$ and $\eta\geq 0$. Then for any $x\in[0,1]$, we have
	\begin{flalign*}
		\begin{split}
			&({\rm i}) \frac{\eta x}{1-x+\gamma x}\leq \frac{\eta}{\gamma};\\
			&({\rm ii}) {\rm min\{\gamma,1\}}\leq 1-x+\gamma x\leq {\rm max\{\gamma,1\}}.
		\end{split}&
	\end{flalign*}
\end{lem}

\section{Prior construction method for Schur complement} \label{section.3}
 For the sake of convenience, we supplement some symbols.  Given $A\in{\mathbb C^{n\times n}}$  and  $S\subseteq N$, denote 
 \begin{align*}
 	R^{S}_{i}(A)=\underset{S\backslash \{i\}}\sum |a_{ij}|, \quad Q^{S}_{i}(A)=\underset{j\in S\backslash \{i\}}\sum |a_{ij}|\frac{R_{j}(A)}{|a_{jj}|}.
 \end{align*}

\noindent   Specifically, $R^{N}_{i}(A)=R_{i}(A)$, and $	P_{i}(A)=R^{N_{1}}_{i}(A)+Q^{N_{2}}_{i}(A)$.   
Take $\emptyset \neq \alpha \subsetneq N$, we always label $\alpha=\{i_{1},i_{2},\dots,i_{k}\}$ and $\overline \alpha=\{j_{1},j_{2},\dots,j_{l}\}$ with
\begin{align*}
	i_{1}<i_{2}<\dots<i_{k},\quad j_{1}<j_{2}<\dots<j_{l},
	\end{align*}
and $k+l=n$. For any $t\in N$, denote
\begin{equation*}
	\begin{aligned}
		\begin{array}{c}
			a_{t\alpha}=(a_{ti_{1}},a_{ti_{2}},\dots,a_{ti_{k}}),\quad a_{\alpha t}=(a_{i_{1}t},a_{i_{2}t},\dots,a_{i_{k}t})^{T},\\
			\\
			|a_{t\alpha}|=(|a_{ti_{1}}|,|a_{ti_{2}}|,\dots,|a_{ti_{k}}|), \quad |a_{\alpha t}|=(|a_{i_{1}t}|,|a_{i_{2}t}|,\dots,|a_{i_{k}t}|)^{T}.
		\end{array}
	\end{aligned}
\end{equation*}
If $A(\alpha)$ is nonsingular, we denote $A/\alpha=(a^{\prime}_{tu})$. Then for any entry $a^{\prime}_{tu}$ of $A/\alpha$, there is an expression
\begin{align*}
	a^{\prime}_{tu}=a_{j_{t}j_{u}}-a_{j_{t}\alpha}A(\alpha)^{-1}a_{\alpha j_{u}}.
\end{align*}
If  $A(\alpha)$ is an  $H$-matrix, we take
\begin{align}
	\Delta_{j_{t}j_{u}}=|a_{j_{t}\alpha}|< A(\alpha)>^{-1}|a_{\alpha j_{u}}|.\label{eq3.1}
\end{align}
By the absolute value inequality and Lemma \ref{Lemma 2.1},  we obtain that 
\begin{align*}
	|a_{j_{t}j_{u}}|-\Delta_{j_{t}j_{u}}&\leq |a_{j_{t}j_{u}}|-|a_{j_{t}\alpha}||A(\alpha)^{-1}||a_{\alpha j_{u}}|\leq|a^{\prime}_{tu}|\\
	&\leq |a_{j_{t}j_{u}}|+|a_{j_{t}\alpha}||A(\alpha)^{-1}||a_{\alpha j_{u}}|\leq|a_{j_{t}j_{u}}|+\Delta_{j_{t}j_{u}} ,
\end{align*}
i.e.,
\begin{align}
	|a_{j_{t}j_{u}}|-\Delta_{j_{t}j_{u}}\leq |a^{\prime}_{tr}|\leq |a_{j_{t}j_{u}}|+\Delta_{j_{t}j_{u}}. \label{eq3.2}
\end{align}
Next, we give an important Lemma that runs through this paper.
\begin{lem}\rm \label{Lemma 3.1} Let  $A=(a_{ij})\in \mathbb C^{n\times n}$,  $\emptyset \neq \alpha\subseteq N_{2}$ and $|\alpha|<n$, $x=(x_{1},x_{2},\dots.x_{k})^{T}$. If, for all $t\in\{1,2,\dots,k\}$, $0\leq x_{t}\leq R^{\overline \alpha}_{i_{t}}(A)$, then
	\begin{align*}
		< A(\alpha)>^{-1}x\leq& (\frac{x_{1}+Q^{\alpha}_{i_{1}}(A)}{|a_{i_{1}i_{1}}|},\frac{x_{2}+Q^{\alpha}_{i_{2}}(A)}{|a_{i_{2}i_{2}}|},\dots.\frac{x_{k}+Q^{\alpha}_{i_{k}}(A)}{|a_{i_{k}i_{k}}|})^{T}\\
		\leq&(\frac{R_{i_{1}}(A)}{|a_{i_{1}i_{1}}|},\frac{R_{i_{2}}(A)}{|a_{i_{2}i_{2}}|},\dots,\frac{R_{i_{k}}(A)}{|a_{i_{k}i_{k}}|})^{T}.
	\end{align*}
	\end{lem}

\begin{proof}
	Let $y=(y_{1},y_{2},\dots, y_{k})^{T}$, where \begin{align*}y_{t}=\frac{x_{t}+Q^{\alpha}_{i_{t}}(A)}{|a_{i_{t}i_{t}}|}\leq \frac{R_{i_{t}}(A)}{|a_{i_{t}i_{t}}|}.
		\end{align*}
	  Since $\emptyset \neq \alpha\subseteq N_{2}$, then $A(\alpha)$ is an SDD matrix, it is an $H$-matrix, thus $< A(\alpha)>^{-1}\geq 0$. we first demonstrate
	\begin{align*}
		x\leq < A(\alpha) > y.
		\end{align*}
		For any $t\in\{1,2,\dots,k\}$, consider $t$-th component of $< A(\alpha)> y$ as follows:
		\begin{align*}
	|a_{i_{t}i_{t}}|y_{t}-\underset{i_{u}\in \alpha\backslash \{i_{t}\}}\sum|a_{i_{t}i_{u}}|y_{u}\geq& x_{t}+Q^{\alpha}_{i_{t}}(A)-\underset{i_{u}\in \alpha\backslash \{i_{t}\}}\sum|a_{i_{t}i_{u}}|\frac{R_{i_{u}}(A)}{|a_{i_{u}i_{u}}|}\\
	=&x_{t}+Q^{\alpha}_{i_{t}}(A)-Q^{\alpha}_{i_{t}}(A)=x_{t}.		
		\end{align*}
	Due to the arbitrariness of $t$, it follows that $	x\leq  < A(\alpha)> y$. According to Lemma \ref{Lemma 2.2}, we  have $<A(\alpha)>^{-1}x\leq y$. The proof is completed.         	
	\end{proof}
	\begin{remark}\rm 
		It can be observed that in Lemma \ref{Lemma 3.1}, the key to the proof lies in constructing a fictitious positive vector $y$  such that $x$ is dominated by $<A(\alpha)>y$. Subsequently, by leveraging the non-negativity of $<A(\alpha)>^{-1}$ and Lemma \ref{Lemma 2.6}, we obtain the final estimate. Furthermore, the tighter the construction of the vector $y$, the more effective the bounding becomes. We refer to this procedure as  “ prior construction".
	\end{remark}

	For the sake of our topic, in the remaining part of this paper, unless otherwise specified, we always assume that the matrix $A$  under discussion satisfies the condition that both $N_{1}$  and  $N_{2}$ are non empty subsets.

\begin{lem}\rm \label{Lemma 3.2}
	Let  $A=(a_{ij})\in \mathbb C^{n\times n}$  with $\emptyset \neq \alpha \subseteq N_{2}$. Then the following hold:\\
	(i)For any $j_{t}\in \overline \alpha$, we have 
	\begin{align*}
 |a^{\prime}_{tt}|-R_{t}(A/\alpha)\geq |a_{j_{t}j_{t}}|-R^{\overline \alpha}_{j_{t}}(A)-Q^{\alpha}_{j_{t}}(A)\geq |a_{j_{t}j_{t}}|-R_{j_{t}}(A).
	\end{align*} 
	(ii) If $\emptyset \neq \alpha \subsetneq N_{2}$, then for any $j_{t}\in N_{2}\backslash \alpha$, we have 
	\begin{align*}
		\frac{R_{t}(A/\alpha)}{|a^{\prime}_{tt}|}\leq \frac{R_{j_{t}}(A)}{|a_{j_{t}j_{t}}|}.
	\end{align*}
\end{lem}
\begin{proof}
(i)	Note that 
	\begin{align*}
		|a^{\prime}_{tt}|-R_{t}(A/\alpha)&=|a^{\prime}_{tt}|-\underset{j_{u}\in \overline \alpha\backslash \{j_{t}\}}\sum|a^{\prime}_{tu}|\\
		&\geq |a_{j_{t}j_{t}}|-\Delta_{j_{t}j_{t}}-\underset{j_{u}\in \overline \alpha\backslash \{j_{t}\}}\sum \left(|a_{j_{t}j_{u}}|+\Delta_{j_{t}j_{u}}\right)\\
		&=|a_{j_{t}j_{t}}|-\underset{j_{u}\in \overline \alpha \backslash \{j_{t}\}}\sum|a_{j_{t}j_{u}}|-\underset{j_{u}\in \overline \alpha}\sum\Delta_{j_{t}j_{u}}.
	\end{align*}
	According to Lemma \ref{Lemma 3.1}, we have
	\begin{align*}
		\underset{j_{u}\in \overline \alpha}\sum\Delta_{j_{t}j_{u}}&=|a_{j_{t}\alpha}|< A(\alpha)>^{-1} \left(R^{\overline \alpha}_{i_{1}}(A),R^{\overline \alpha}_{i_{2}}(A),\dots,R^{\overline \alpha}_{i_{k}}(A)\right)^{T}\\
		&\leq |a_{j_{t}\alpha}|\left(\frac{R_{i_{1}}(A)}{|a_{i_{1}i_{1}}|},\frac{R_{i_{2}}(A)}{|a_{i_{2}i_{2}}|},\dots \frac{R_{i_{k}}(A)}{|a_{i_{k}i_{k}}|}\right)^{T}\\
		&=Q^{\alpha}_{j_{t}}(A).
	\end{align*}
	Therefore 
	\begin{align}
		|a^{\prime}_{tt}|-R_{t}(A/\alpha)\geq& |a_{j_{t}j_{t}}|-\underset{j_{u}\in \overline \alpha \backslash \{j_{t}\}}\sum|a_{j_{t}j_{u}}|-\underset{j_{u}\in \overline \alpha}\sum\Delta_{j_{t}j_{u}}\notag\\
		\geq &  |a_{j_{t}j_{t}}|-R^{\overline \alpha}_{j_{t}}(A)-Q^{\alpha}_{j_{t}}(A)\notag\\
		\geq& |a_{j_{t}j_{t}}|-R_{j_{t}}(A).\label{eq 3.3}
	\end{align}
	(ii) For  any $j_{t}\in N_{2}\backslash \alpha$,  owing to (\ref{eq 3.3}), it follows that
	\begin{align}
		 |a_{j_{t}j_{t}}|-\underset{j_{u}\in \overline \alpha \backslash \{j_{t}\}}\sum|a_{j_{t}j_{u}}|-\underset{j_{u}\in \overline \alpha}\sum\Delta_{j_{t}j_{u}}\geq |a_{j_{t}j_{t}}|-R_{j_{t}}(A)>0.\label{eq 3.4}
	\end{align}
Therefore, leveraging  Lemma \ref{Lemma 3.1}, Lemma \ref{Lemma 2.5} and (\ref{eq 3.4}), we can derive the following:
\begin{align*}
	\frac{R_{t}(A/\alpha)}{|a^{\prime}_{tt}|}=\frac{\underset{j_{u}\in \overline \alpha \backslash \{j_{t}\}}\sum|a^{\prime}_{tu}|}{|a^{\prime}_{tt}|}\leq \frac{\underset{j_{u}\in \overline \alpha\backslash \{j_{t}\}}\sum(|a_{j_{t}j_{u}}|+\Delta_{j_{t}j_{u}})}{|a_{j_{t}j_{t}}|-\Delta_{j_{t}j_{t}}}\leq \frac{\underset{j_{u}\in \overline \alpha \backslash \{j_{t}\}}\sum|a_{j_{t}j_{u}}|+\underset{j_{u}\in \overline \alpha}\sum\Delta_{j_{t}j_{u}}}{|a_{j_{t}j_{t}}|}\leq \frac{R_{j_{t}}(A)}{|a_{j_{t}j_{t}}|}.
\end{align*}	
The proof is completed.
\end{proof}

To continue our main theorem, we add the following definition:
	\begin{align*}
	\widetilde {N_{1}}:=\{j_{t}\in \overline \alpha:|a^{\prime}_{tt}|\leq R_{t}(A/\alpha)\},\quad \widetilde {N_{2}}:=	\{j_{t}\in \overline \alpha:|a^{\prime}_{tt}|> R_{t}(A/\alpha)\}.
	\end{align*}
It can be found that $\widetilde {N_{1}}$ describes the non-diagonally dominant set of $A/\alpha$ in terms of the elements of the original set, and correspondingly $\widetilde {N_{2}}$ describes the diagonally dominant set of $A/\alpha$ in terms of the elements of the original set.  To elucidate this concept, we present the following  example.

\begin{example}\label{Example4.1}\rm Consider the following matrix
	\begin{align*}
		A=\left(
		\begin{array}{ccccc}
			2& 0& 1 & 0&0\\
			 1&  5&   1& 1& 0  \\
			0 & 0 & 3 &0  &0 \\
			1 & 0 &0 &2 &1  \\
			0 & 2 & 0 &1& 2 \\
		\end{array}
		\right)  .	
	\end{align*}
We have $N_{1}=\{4,5\}$, $N_{2}=\{1,2,3\}$. Take $\alpha=\{1\}$, then\vspace{-11ex}
$$
\adjustbox{valign=B}{%
	$A/\alpha=\left(
	\begin{array}{cccc}
		5 & 0.5 & 1 & 0 \\
		0 & 3 & 0 & 0 \\
		0 & -0.5 & 2 & 1 \\
		2 & 0 & 1 & 2 \\
	\end{array}
	\right)$%
}
\quad 
\adjustbox{valign=B}{%
	$\begin{aligned}
		 \vphantom{\begin{array}{c}0\\0\\0\\0\end{array}} \\[-0.8ex]
		 &\dashrightarrow\\
		&\dashrightarrow \\[-0.8ex]
		&\dashrightarrow \\[-0.8ex]
		&\dashrightarrow
	\end{aligned}$%
}
\quad 
\adjustbox{valign=B}{%
	$\left(
	\begin{array}{c:c c c c}
		2 & 0 & 1 & 0 & 0 \\
		\hdashline
		1 & 5 & 1 & 1 & 0 \\      
		0 & 0 & 3 & 0 & 0 \\      
		1 & 0 & 0 & 2 & 1 \\      
		0 & 2 & 0 & 1 & 2 \\      
	\end{array}
	\right)$.%
}
$$
Then there is $	\widetilde {N_{1}}=\{5\}$, $\widetilde {N_{2}}=\{2,3,4\}$.
\end{example}
The following result establishes the quantitative relationship between $\widetilde {N_{1}}$ and $\widetilde {N_{2}}$ when $\alpha$ is a  nonempty proper subset of $N_{2}$.

\begin{lem}\rm \label{Lemma 3.3} Let  $A=(a_{ij})\in  \mathbb C^{n\times n}$, and $\emptyset \neq \alpha\subsetneq N_{2}$. Then, the following properties hold:\\
	(i) $N_{2}\backslash \alpha\subseteq \widetilde {N_{2}}$.\\
	(ii) $\widetilde {N_{1}}\subseteq N_{1}$.\\
	(iii) $N_{1}\backslash \widetilde {N_{1}}=\widetilde {N_{2}}\backslash (N_{2}\backslash \alpha).    $
\end{lem}
\begin{proof}
	Based on Lemma  \ref{Lemma 3.2}, it is evident that  (i) and (ii) hold true. We give a proof for (iii).\\
		($\implies$) Suppose $N_1 \setminus \widetilde{N_1} \neq \emptyset$. For any $x \in N_1 \setminus \widetilde{N_1}$, we have:
	\begin{itemize}[leftmargin=*,nosep]
		\item $x \in N_1$ by membership condition;
		\item $x \notin \widetilde{N_1}$ by set difference definition;
		\item $x \in \widetilde{N_2}$ through complementary relationship.
	\end{itemize}
	If $x \in N_2 \setminus \alpha$, this would contradict $x \in N_1$ since $N_1 \cap N_2 = \emptyset$. Therefore
	\[
	x \in \widetilde{N_2} \setminus (N_2 \setminus \alpha).
	\]
	By the arbitrariness of $x$, we establish
	\begin{align}
		N_1 \setminus \widetilde{N_1} \subseteq \widetilde{N_2} \setminus (N_2 \setminus \alpha). \label{eq:subset1}
	\end{align}
	The containment trivially holds when $N_1 \setminus \widetilde{N_1} = \emptyset$.\\
($\impliedby$) Assume $\widetilde{N_2} \setminus (N_2 \setminus \alpha) \neq \emptyset$. For any $x \in \widetilde{N_2} \setminus (N_2 \setminus \alpha)$, we deduce
	\begin{itemize}[leftmargin=*,nosep]
		\item $x \in \widetilde{N_2}$ by membership condition;
		\item $x \notin N_2 \setminus \alpha$ by set difference definition;
		\item $x \in N_1$ by $x\in \overline\alpha$ and $N_{1}\cap N_{2}=\emptyset $.			
	\end{itemize}
	If $x\in  \widetilde{N_1}$, this would contradict $x\in \widetilde{N_2}$ since $\widetilde{N_1}\cap \widetilde{N_2}=\emptyset$.		This implies
	\[
	x \in N_1 \setminus \widetilde{N_1}.
	\]
	By the arbitrariness of $x$, we obtain that
	\begin{align}
		\widetilde{N_2} \setminus (N_2 \setminus \alpha) \subseteq N_1 \setminus \widetilde{N_1}. \label{eq:subset2}
	\end{align}
	The containment trivially holds when $\widetilde{N_2} \setminus (N_2 \setminus \alpha) = \emptyset$.
	Combining \eqref{eq:subset1} and \eqref{eq:subset2}, we conclude
	\[
	N_1 \setminus \widetilde{N_1} = \widetilde{N_2} \setminus (N_2 \setminus \alpha). \qedhere.
	\]
The proof is completed.	
\end{proof}

\begin{figure}[h]
	\begin{minipage}{1.0 \linewidth}
		\centering
		\includegraphics[width=4in]{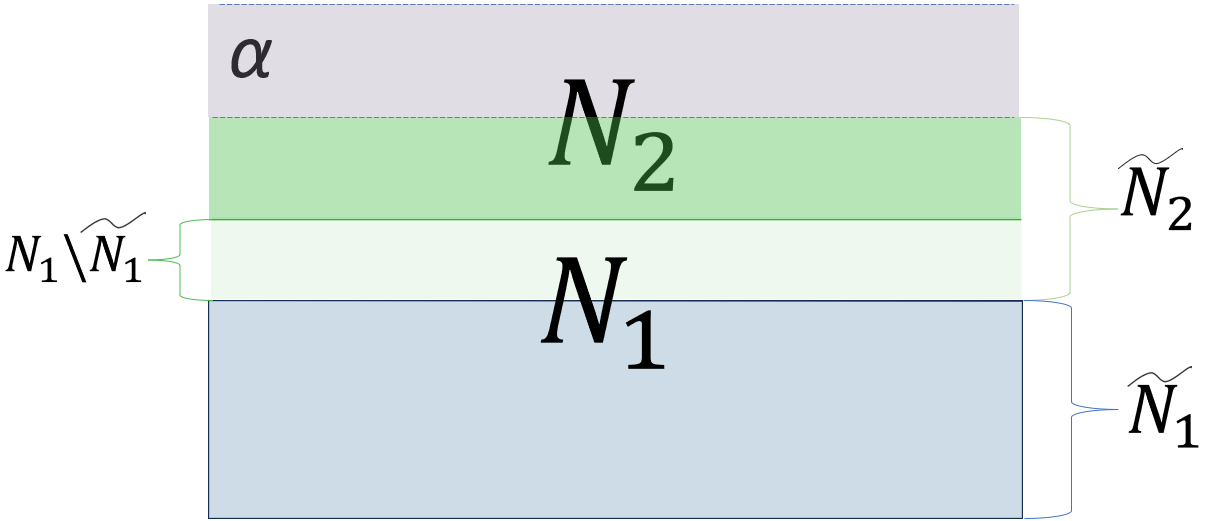}
	\end{minipage}
\qquad	\caption{The relation of $N_{1}$, $N_{2}$, $\widetilde N_{1}$, $\widetilde N_{2}$. }\label{Figure.1}
\end{figure}

We graphically express Lemma \ref{Lemma 3.3} with Figure. \ref{Figure.1}. Lemma \ref{Lemma 3.3} tells us that when a subset $\emptyset \neq\alpha\subsetneq N_{2}$ is taken, the Schur complement  $A/\alpha$ may cause the original non-diagonally dominant row to become diagonally dominant, and that the increase of the diagonally dominant row is equal to the decrease of the non-diagonally dominant row.
\par
Next, after doing enough preparatory work, we give one of the three main results of this Section.
\begin{theorem}\rm  \label{Theorem 3.1} Let $A=(a_{ij})\in \mathbb C^{n\times n}$ be an $SDD_{1}$ matrix. If $\emptyset \neq \alpha \subsetneq N_{2}$, then for any $j_{t}\in \overline \alpha$, we have
	\begin{align}
		|a^{\prime}_{tt}|-P_{t}(A/\alpha)&\geq |a_{j_{t}j_{t}}|-R^{N_{1}}_{j_{t}}(A)-Q^{N_{2}\backslash \alpha}_{j_{t}}(A)-\underset{h\in \alpha}\sum\frac{|a_{j_{t}h}|}{|a_{hh}|}\left(R^{N_{1}\cup \{j_{t}\}}_{h}(A)+Q^{N_{2}\backslash \{j_{t}\}}_{h}(A)\right)\notag\\
		&\geq |a_{j_{t}j_{t}}|-P_{j_{t}}(A)>0.\label{eq 3.5}
	\end{align}
\end{theorem}
\begin{proof}
	By Lemma \ref{Lemma 3.3}, we can suppose  $N_{1}\backslash \widetilde {N_{1}}=\widetilde {N_{2}}\backslash (N_{2}\backslash \alpha)=L$. Therefore 
	\begin{align*}
		|a^{\prime}_{tt}|-P_{t}(A/\alpha)&=|a^{\prime}_{tt}|-\underset{j_{u}\in \widetilde N_{1}\backslash \{j_{t}\}}\sum|a^{\prime}_{tu}|-\underset{j_{u}\in \widetilde N_{2}\backslash \{j_{t}\}}\sum|a^{\prime}_{tu}|\frac{R_{u}(A/\alpha)}{|a^{\prime}_{uu}|}\\
		&\geq |a^{\prime}_{tt}|-\underset{j_{u}\in \widetilde N_{1}\backslash \{j_{t}\}}\sum|a^{\prime}_{tu}|-\underset{j_{u}\in L\backslash\{j_{t}\} }\sum|a^{\prime}_{tu}|-\underset{j_{u}\in \widetilde N_{2}\backslash L \backslash \{j_{t}\}}\sum|a^{\prime}_{tu}|\frac{R_{u}(A/\alpha)}{|a^{\prime}_{uu}|}\\
		&=|a^{\prime}_{tt}|-\underset{j_{u}\in N_{1}\backslash \{j_{t}\}}\sum|a^{\prime}_{tu}|-\underset{j_{u}\in N_{2}\backslash\alpha\backslash \{j_{t}\} }\sum|a^{\prime}_{tu}|\frac{R_{u}(A/\alpha)}{|a^{\prime}_{uu}|}\\
		&\geq |a^{\prime}_{tt}|-\underset{j_{u}\in N_{1}\backslash \{j_{t}\}}\sum|a^{\prime}_{tu}|-\underset{j_{u}\in N_{2}\backslash\alpha\backslash \{j_{t}\} }\sum|a^{\prime}_{tu}|\frac{R_{j_{u}}(A)}{|a_{j_{u}j_{u}}|}(\text{\rm Form  (ii)  of  Lemma \ref{Lemma 3.2}})\\
		&\geq |a_{j_{t}j_{t}}|-\Delta_{j_{t}j_{t}}-\underset{j_{u}\in N_{1}\backslash \{j_{t}\}}\sum(|a_{j_{t}j_{u}}|+\Delta_{j_{t}j_{u}})-\underset{j_{u}\in N_{2}\backslash\alpha\backslash \{j_{t}\} }\sum(|a_{j_{t}j_{u}}|+\Delta_{j_{t}j_{u}})\frac{R_{j_{u}}(A)}{|a_{j_{u}j_{u}}|}\\
		&=|a_{j_{t}j_{t}}|-R^{N_{1}}_{j_{t}}(A)-Q^{N_{2}\backslash \alpha}_{j_{t}}(A)- \left(\Delta_{j_{t}j_{t}}+\underset{j_{u}\in N_{1}\backslash \{j_{t}\}}\sum\Delta_{j_{t}j_{u}}+\underset{j_{u}\in N_{2}\backslash\alpha\backslash \{j_{t}\} }\sum\Delta_{j_{t}j_{u}}\frac{R_{j_{u}}(A)}{|a_{j_{u}j_{u}}|}\right).
	\end{align*}
Make use of Lemma \ref{Lemma 3.1}, we have	
	\begin{align*}
		&\Delta_{j_{t}j_{t}}+\underset{j_{u}\in N_{1}\backslash \{j_{t}\}}\sum\Delta_{j_{t}j_{u}}+\underset{j_{u}\in N_{2}\backslash\alpha\backslash \{j_{t}\} }\sum\Delta_{j_{t}j_{u}}\frac{R_{j_{u}}(A)}{|a_{j_{u}j_{u}}|}\\
		=&|a_{j_{t}\alpha}|< A(\alpha)>^{-1}\left(R^{N_{1}\cup \{j_{t}\}}_{i_{1}}(A)+Q^{N_{2}\backslash \alpha\backslash \{j_{t}\}}_{i_{1}}(A),\dots,  R^{N_{1}\cup \{j_{t}\}}_{i_{k}}(A)+Q^{N_{2}\backslash \alpha\backslash \{j_{t}\}}_{i_{k}}(A)\right)^{T}\\
		\leq& \underset{h\in \alpha}\sum\frac{|a_{j_{t}h}|}{|a_{hh}|}\left(R^{N_{1}\cup \{j_{t}\}}_{h}(A)+Q^{N_{2}\backslash \{j_{t}\}}_{h}(A)\right)\leq Q^{\alpha}_{j_{t}}(A).
	\end{align*}
	Therefore, we have
	\begin{align*}
			|a^{\prime}_{tt}|-P_{t}(A/\alpha)&\geq |a_{j_{t}j_{t}}|-R^{N_{1}}_{j_{t}}(A)-Q^{N_{2}\backslash \alpha}_{j_{t}}(A)-\underset{h\in \alpha}\sum\frac{|a_{j_{t}h}|}{|a_{hh}|}\left(R^{N_{1}\cup \{j_{t}\}}_{h}(A)+Q^{N_{2}\backslash \{j_{t}\}}_{h}(A)\right)\\
		&\geq |a_{j_{t}j_{t}}|-R^{N_{1}}_{j_{t}}(A)-Q^{N_{2}\backslash \alpha}_{j_{t}}(A)    -Q^{\alpha}_{j_{t}}(A)                                             =|a_{j_{t}j_{t}}|-P_{j_{t}}(A)>0
	\end{align*}
	The proof is completed.
\end{proof}
\begin{remark}\rm Given a matrix $A\in \mathbb C^{n\times n}$, we denote the $D_{1}$-diagonally  dominant degree as
	\begin{align*}
	 \delta_{i}=|a_{ii}|-P_{i}(A),          \qquad i\in N.  
	 \end{align*}
	 Then,  $A$ is an $SDD_{1}$ matrix if and only if all the $D_{1}$-diagonally dominant degrees of $A$ are positive. Theorem \ref{Theorem 3.1} tells us if  $A$ is  an $SDD_{1}$ matrix and $\alpha$ is a nonempty proper subset of $N_{2}$, then $A/\alpha$  is still an $SDD_{1}$ matrix.  Moreover, the $D_{1}$-diagonally dominant degree of its Schur complement is better than the $D_{1}$-diagonally dominant degree of the original matrix in the corresponding row.
	\end{remark}

\begin{example}\rm Consider the $SDD_{1}$ matrix
		\begin{align*}
		A=\left(
		\begin{array}{cccccc}
			10& 0& 1 & 0&2  &3  \\
			1&  16&   1& 2& 0 &0  \\
			2 & 0 & 20 &0  &2  &1\\
			1 & 7 &3 &12 &3 &1 \\
			2 & 3 & 10 &1& 7 &0\\
			6& 3 & 16& 7& 5&22
		\end{array}
		\right)  .	
	\end{align*}
We have $N_{1}=\{4,5, 6\}$, $N_{2}=\{1,2,3\}$. 	 Take $\alpha=\{1,2\}$, then
	\begin{align*}
	A/\alpha=(a^{\prime}_{tu})=\left(
	\begin{array}{cccccc}
		19.8& 0& 1.6 & 0.4  \\
		2.5063&  11.125&   2.8875& 0.8313 \\
		9.6313 & 0.625 & 6.6375 &-0.5438  \\
		15.2313 & 6.625 &3.8375 &20.2563 \\
	\end{array}
	\right).	
\end{align*}
	Note that the non-diagonally dominant set of $A/\alpha$ is $\{3,4\}$ and diagonally dominant set of $A/\alpha$ is $\{1,2\}$, and
	\begin{align*}
		&|a^{\prime}_{11}|-P_{1}(A/\alpha)=19.8-2=17.8,\\
		&|a^{\prime}_{22}|-P_{2}(A/\alpha)=11.125-2.5063\times \frac{2}{19.8}-2.8875-0.8313=7.153,\\
  	&|a^{\prime}_{33}|-P_{3}(A/\alpha)=6.6375-9.6313\times \frac{2}{19.8} -0.625\times\frac{6.2251}{11.125}-0.5438=4.7711,\\ 
  	&|a^{\prime}_{44}|-P_{4}(A/\alpha)=20.2563-15.2313\times\frac{2}{19.8}-6.625\times\frac{6.2251}{11.125}-3.8375=11.1732.
	\end{align*}
	Then $A/\alpha$ is an $SDD_{1}$ matrix.
	\end{example}

The following theorem reveals that when $\alpha$ increases from a proper subset of $N_{2}$ to $N_{2}$ itself, there is a qualitative change in the properties of $A/\alpha$, and it becomes an SDD matrix.

\begin{theorem}\rm \label{Theorem 3.2} Let $A=(a_{ij})\in \mathbb C^{n\times n}$ be an $SDD_{1}$ matrix. If $\alpha=N_{2}$, then for any $j_{t}\in \overline\alpha$, we have
	\begin{align}
		|a^{\prime}_{tt}|-R_{t}(A/\alpha)\geq |a_{j_{t}j_{t}}|-R^{N_{1}}_{j_{t}}(A)-\underset{h\in N_{2}}\sum\frac{|a_{j_{t}h}|}{|a_{hh}|}P_{h}(A)\geq |a_{j_{t}j_{t}}|-P_{j_{t}}(A)>0.\label{eq 3.6}
	\end{align}
\end{theorem}
\begin{proof}
	Note that $\alpha=N_{2}$, then
	\begin{align*}
	|a^{\prime}_{tt}|-R_{t}(A/\alpha)&=|a^{\prime}_{tt}|-\underset{j_{u}\in N_{1}\backslash \{j_{t}\}}\sum|a^{\prime}_{tu}|\\
	&\geq |a_{j_{t}j_{t}}|-R^{N_{1}}_{j_{t}}(A)-\underset{j_{u}\in N_{1}}\sum\Delta_{j_{t}j_{u}}.
	\end{align*}
	Utilize Lemma \ref{Lemma 3.1}, it follows that
	\begin{align*}
		\underset{j_{u}\in N_{1}}\sum\Delta_{j_{t}j_{u}}&=|a_{j_{t}\alpha}|< A(\alpha) \notag >^{-1}\left(R^{N_{1}}_{i_{1}}(A),R^{N_{1}}_{i_{2}}(A),\dots,R^{N_{1}}_{i_{k}}(A)\right)^{T}\leq \underset{h\in N_{2}}\sum \frac{|a_{j_{t}h}|}{|a_{hh}|}P_{h}(A).\\ 
	\end{align*}
	Therefore \vspace{-1.5ex}
	\begin{align*}
		|a^{\prime}_{tt}|-R_{t}(A/\alpha)\geq  |a_{j_{t}j_{t}}|-R^{N_{1}}_{j_{t}}(A)- \underset{h\in N_{2}}\sum \frac{|a_{j_{t}h}|}{|a_{hh}|}P_{h}(A)\geq  |a_{j_{t}j_{t}}|-P_{j_{t}}(A)>0.
		\end{align*}
		The proof is completed.
\end{proof}

\begin{cor}\rm  Let $A=(a_{ij})\in \mathbb C^{n\times n}$ be an $SDD_{1}$ matrix. If $N_{2}\subseteq \alpha \subsetneq N$, then $A/\alpha$ is an SDD matrix.
\end{cor}
\begin{proof}
	When $\alpha=N_{2}$, by Theorem \ref{Theorem 3.2}, we have $|a^{\prime}_{tt}|-R_{t}(A/\alpha)>0$ for any $j_{t}\in \overline \alpha$ holds, by definition of SDD matrix, the $A/\alpha$ is an SDD matrix. When $N_{2}\subsetneq \alpha \subsetneq N$, we note the fact that the Schur complement of SDD matrix remains SDD matrix \cite{Liu SIAM}. Meanwhile, $A(\alpha)/N_{2}$ is a principal submatrix  of $A/N_{2}$. Since $A/N_{2}$ is already an SDD matrix and by Lemma \ref{Lemma 2.3}, it follows that
	\begin{align*}
		A/\alpha=(A/N_{2})/(A(\alpha)/N_{2})
	\end{align*}
	is an SDD matrix. The proof is completed.
\end{proof}
\begin{remark}\rm
	It is a classical and commonly used method to discuss the Schur complement property of $H$-matrix by using Quotient formula, but it has a significant shortcoming,  it can not get exact numerical results.  Recalling Theorem \ref{Theorem 3.2}, we know that when $\alpha =N_{2}$. Then, for any $j_{t}\in \overline \alpha$ 
	\begin{align*}
			|a^{\prime}_{tt}|-R_{t}(A/\alpha)\geq |a_{j_{t}j_{t}}|-R^{N_{1}}_{j_{t}}(A)-\underset{h\in N_{2}}\sum\frac{|a_{j_{t}h}|}{|a_{hh}|}P_{h}(A)>0.
	\end{align*}
	 We naturally have a question: When $N_{2}\subsetneq \alpha \subsetneq N$,what is the accurate positive lower bound for $|a^{\prime}_{tt}|-R_{t}(A/\alpha)$?  Next, we use the prior construction method to give precise numerical results to answer this question.
	\end{remark}
	
\begin{theorem}\rm 
 Let $A=(a_{ij})\in \mathbb C^{n\times n}$ be an $SDD_{1}$ matrix. If $N_{2}\subsetneq \alpha \subsetneq N$, then for any $j_{t}\in \overline\alpha$, we have
 \begin{align}
 	|a^{\prime}_{tt}|-R_{t}(A/\alpha)\geq |a_{j_{t}j_{t}}|-R^{\overline \alpha}_{j_{t}}(A)-\underset{h\in \alpha}\sum\frac{|a_{j_{t}h}|}{|a_{hh}|}P_{h}(A)\geq |a_{j_{t}j_{t}}|-P_{j_{t}}(A)>0.\label{eq 3.7}
 \end{align}

\end{theorem}
\begin{proof}
	Without loss generality, we suppose $N_{2}=\{i_{1},i_{2},\dots, i_{m}\}$ and $\alpha \cap N_{1}=\{i_{m+1},i_{m+2},\dots, i_{k}\}$. Therefore,
\begin{align}
	|a^{\prime}_{tt}|-R_{t}(A/\alpha)&=|a^{\prime}_{tt}|-\underset{j_{u}\in \overline \alpha \backslash \{j_{t}\}}\sum |a^{\prime}_{tu}| \notag\\ 
	&\geq |a_{j_{t}j_{t}}|-R^{\overline \alpha }_{j_{t}}(A)-\underset{j_{u}\in \overline \alpha}\sum \Delta_{j_{t}j_{u}}, \label{Formula 3}
\end{align}
	where 
	\begin{align}
		\underset{j_{u}\in \overline \alpha}\sum \Delta_{j_{t}j_{u}}=|a_{j_{t}\alpha}|< A(\alpha) >^{-1}\left(R^{\overline \alpha}_{i_{1}}(A),\dots, R^{\overline \alpha}_{i_{m}}(A),  R^{\overline \alpha}_{i_{m+1}}(A),\dots, R^{\overline \alpha}_{i_{k}}(A)        \right)^{T}.\label{eq 3.9}
	\end{align}
	Next, we prove that 
	\begin{align}
		&< A(\alpha)>^{-1}\left(R^{\overline \alpha}_{i_{1}}(A),\dots, R^{\overline \alpha}_{i_{m}}(A),  R^{\overline \alpha}_{i_{m+1}}(A),\dots, R^{\overline \alpha}_{i_{k}}(A)\right)^{T}\notag\\
		\leq& \left(\frac{P_{i_{1}}(A)}{|a_{i_{1}i_{1}}|},\frac{P_{i_{2}}(A)}{|a_{i_{2}i_{2}}|},\dots,\frac{P_{i_{m}}(A)}{|a_{i_{m}i_{m}}|},\frac{P_{i_{m+1}}(A)}{|a_{i_{m+1}i_{m+1}}|},\dots,\frac{P_{i_{k}}(A)}{|a_{i_{k}i_{k}}|}    \right)^{T}\leq \left(\frac{R_{i_{1}}(A)}{|a_{i_{1}i_{1}|}},\dots,\frac{R_{i_{m}}(A)}{|a_{i_{m}i_{m}|}},1,\dots,1\right)^{T}.\label{eq 3.10}
	\end{align}
	The inequality to the right of (\ref{eq 3.10}) is obviously valid. We first prove
	\begin{align}
	\left(R^{\overline \alpha}_{i_{1}}(A), R^{\overline \alpha}_{i_{2}}(A), \dots, R^{\overline \alpha}_{i_{k}}(A)\right)^{T}\leq < A(\alpha) > y\label{eq 3.11},
	\end{align}		
	where $y=(y_{1},y_{2},\dots, y_{k})$ and $y_{t}=\frac{P_{i_{t}}(A)}{|a_{i_{t}i_{t}}|}$. Consider $t$-th of $< A(\alpha) > y $ as follow as:
	\begin{align*}
	|a_{i_{t}i_{t}}|y_{t}-\underset{u \neq t}{\sum_{u=1}^{k}} |a_{i_{t}i_{u}}|y_{u}&\geq P_{i_{t}}(A)-\underset{u \neq t}{\sum_{u=1}^{m}}|a_{i_{t}i_{u}}| \frac{P_{i_{u}}(A)}{|a_{i_{u}i_{u}}|}-\underset{u \neq t}{\sum_{u=m+1}^{k}}|a_{i_{t}i_{u}}|\\
	&\geq P_{i_{t}}(A)-Q^{N_{2}}_{i_{t}}(A)-R^{\alpha  \cap N_{1}}_{i_{t}}(A)=R^{\overline \alpha}_{i_{t}}(A).
	\end{align*}
	 Due to the arbitrariness of $t$, we have  (\ref{eq 3.11}) holds. Note that any principal submatrix of an $H$-matrix is still an $H$-matrix \cite{Zfz}. Hence $< A(\alpha)>^{-1}\geq 0$.   By the Lemma \ref{Lemma 2.2}, it follows that (\ref{eq 3.10}) holds.
	Therefore, we have
	\begin{align}
			\underset{j_{u}\in \overline \alpha}\sum \Delta_{j_{t}j_{u}}\leq \underset{h\in \alpha}\sum \frac{|a_{j_{t}h}|}{|a_{hh}|}P_{h}(A)\leq Q^{N_{2}}_{j_{t}}(A)+R^{\alpha \cap N_{1}}_{j_{t}}(A) . \label{eq 3.12}
	\end{align}
	By (\ref{Formula 3}) and (\ref{eq 3.12}), we have (\ref{eq 3.7}). The proof is completed.
\end{proof}

\section{New upper bound of $	\left\|A^{-1}\right\|_{\infty}$  }\label{section.4}
The condition number $\kappa(A)=	{\left\|A\right\|_{\infty}}	\left\|A^{-1}\right\|_{\infty} $   is commonly employed to assess whether a linear system   $Ax=b$ is ill-conditioned or well-behaved, which involves calculating $	\left\|A^{-1}\right\|_{\infty}$.  Moreover, in the linear complementarity problem,  $	\left\|A^{-1}\right\|_{\infty}$ is used to measure the magnitude of the error bound.  In this section, we will give an infinity norm upper bound of the inverse of the $SDD_{1}$ matrix. Let us review some of the conclusions that already exist.

\begin{lem}\rm \cite{Liu1} Let $A=(a_{ij})\in \mathbb C^{n\times n} (n\geq 2)$ be an SDD matrix, then 
	\begin{align*}
		\left\|A^{-1}\right\|_{\infty}\leq   \underset{\underset{i\neq j}{i,j\in N}}{\rm max}\frac{|a_{jj}|+R_{i}(A)}{|a_{ii}||a_{jj}|-R_{i}(A)R_{j}(A)}
	\end{align*}\label{Lemma 5.1}
	\end{lem}
	
	\begin{lem}\rm  \cite{Yong} \label{Lemma 5.2}
		Let $A=(a_{ij})\in \mathbb C^{n\times n}$ be an SDD matrix, and $B=(b_{ij})\in \mathbb C^{n\times m}$. Then 
		\begin{align*}
			\left\|A^{-1}B\right\|_{\infty} \leq \underset{i\in N}{\rm max} \frac{\sum_{j=1}^{m}|b_{ij}|}{|a_{ii}|-R_{i}(A)}.
		\end{align*}
	\end{lem}

\begin{lem}\rm \cite{Baoji} \label{Lemma 5.3}
	Let $A=(a_{ij})\in \mathbb C^{n\times n} (n\geq 2)$ be an $SDD_{1}$ matrix, then
	\begin{align*}
			\left\|A^{-1}\right\|_{\infty}\leq \frac{{\rm max}\{1,\underset{i\in N_{2}}{\rm max}\frac{p_{i}(A)}{|a_{ii}|}+\epsilon  \}}                          {{\rm min }\{\underset{i\in N_{1}}{\rm min}H_{i},   \underset{i\in N_{2}}{\rm min}Q_{i} \}                 },
		\end{align*} \label{Lemma 4.2}
where
\begin{align*}
	H_{i}=|a_{ii}|-R^{N_{1}}_{i}(A)-\underset{j\in N_{2}\backslash \{i\}}\sum |a_{ij}|\left(\frac{p_{j}(A)}{|a_{jj}|}+\epsilon \right), \quad i\in N_{1},
\end{align*}	
\begin{align*}
	Q_{i}=\epsilon\left(|a_{ii}|-R^{N_{2}}_{i}(A)\right)+\underset{j\in N_{2}\backslash \{i\}}\sum|a_{ij}|\frac{R_{j}(A)-P_{j}(A)}{|a_{jj}|}, \quad i\in N_{2}.
\end{align*}
	The parameter $\epsilon$ need satisfy $	0<\epsilon<\underset{i\in N}{\rm min}\frac{|a_{ii}|-P_{i}(A)}{R^{N_{2}}_{i}(A)}$.
\end{lem}

Below, we review the use of the Schur complement to represent the inverse of a matrix.	Let $A=(a_{ij})\in \mathbb C^{n\times n}$ is an nonsingular matrix, if
		\begin{align*}
		A=\left(
		\begin{array}{cc}
			A(\alpha)&A(\alpha,\overline \alpha)\\
			A(\overline \alpha,\alpha)&A(\overline \alpha)\\
		\end{array}
		\right),
	\end{align*}
	where $A(\alpha)$ is an nonsingular matrix.  Then, 	there exists matrices $X, Y\in \mathbb C^{n\times n}$ such that
	\begin{align*}
		XAY=\left(
		\begin{array}{cc}
			A(\alpha)&O  \\
			O &A/\alpha   \\
		\end{array}
		\right)=:\Lambda,
	\end{align*}
	where 
	\begin{align}
		X=\left(
		\begin{array}{cc}
			I_{|\alpha|}&O  \\
			-	A(\overline \alpha, \alpha)A(\alpha)^{-1} &I_{n-|\alpha|}   \\
		\end{array}
		\right) \quad {\text{and}} \quad 	Y=\left(
		\begin{array}{cc}
			I_{|\alpha|}&-A(\alpha)^{-1}A(\alpha, \overline \alpha)  \\
			O &I_{n-|\alpha|}   \\
		\end{array}
		\right).\label{XY4.1}
	\end{align}
	Therefore, we obtain that
	\begin{align*}
		A^{-1}=Y\left(
		\begin{array}{cc}
			A(\alpha)^{-1}&O  \\
			O &(A/\alpha)^{-1}   \\
		\end{array}
		\right)X=Y\Lambda^{-1}X.
	\end{align*}
 By utilizing the fundamental properties of infinity norm, we have: 
	\begin{align*}
			\left\|A^{-1}\right\|_{\infty}=	\left\|Y\Lambda^{-1}X\right\|_{\infty}\leq 	\left\|Y\right\|_{\infty}	\left\|\Lambda^{-1} X\right\|_{\infty},
	\end{align*}
	where 
	\begin{align}
		\Bigl\| \Lambda^{-1} X \Bigr\|_{\infty}
		&= \Biggl\| 
		\begin{pmatrix}
			A(\alpha)^{-1} & O \notag\\
			O & (A/\alpha)^{-1}
		\end{pmatrix}
		\begin{pmatrix}
			I_{|\alpha|} & O \\
			-A(\overline{\alpha},\alpha) A(\alpha)^{-1} & I_{n-|\alpha|}
		\end{pmatrix}
		\Biggr\|_{\infty} \\
		&= \Biggl\| 
		\begin{pmatrix}
			A(\alpha)^{-1} & O \\
			(A/\alpha)^{-1} \bigl( -A(\overline{\alpha},\alpha) A(\alpha)^{-1} \bigr) & (A/\alpha)^{-1}
		\end{pmatrix}
		\Biggr\|_{\infty} \notag \\ 
		&\leq \max \Bigl\{ 
		\bigl\| A(\alpha)^{-1} \bigr\|_{\infty},\,
		\bigl\| (A/\alpha)^{-1} \bigl( -A(\overline{\alpha},\alpha) A(\alpha)^{-1}\thinspace I_{n-|\alpha|} \bigr) \bigr\|_{\infty} 
		\Bigr\}.\label{ZJFS}
	\end{align}
Moreover, the following holds:
\begin{align*}
	\Bigl\| (A/\alpha)^{-1} \bigl( -A(\overline{\alpha},\alpha) A(\alpha)^{-1}\thinspace I_{n-|\alpha|} \bigr) \Bigr\|_{\infty}
	&\leq \Bigl\| (A/\alpha)^{-1} \Bigr\|_{\infty} \Bigl\| -A(\overline{\alpha},\alpha) A(\alpha)^{-1}\thinspace I_{n-|\alpha|} \Bigr\|_{\infty} \\
	&= \Bigl\| (A/\alpha)^{-1} \Bigr\|_{\infty} \| X \|_{\infty}.
\end{align*}
	Finally, we obtain that
	\begin{align}
			\left\|A^{-1}\right\|_{\infty}\leq 	\left\|Y\right\|_{\infty} {\rm max}\{	\left\|A(\alpha)^{-1}\right\|_{\infty},  \left\|(A/\alpha)^{-1}\right\|_{\infty}\left\|X\right\|_{\infty}\}.\label{Schur infinity 2}
	\end{align}

 Under the idea of splitting and shrinking based on  (\ref{ZJFS}), we employ  Theorem \ref{Theorem 3.2} with Lemma \ref{Lemma 5.1} and Lemma \ref{Lemma 5.2},  give the following upper bound for the infinity norm of the inverse of the $SDD_{1}$ matrix.

\begin{theorem}\label{Theorem 6}\rm
	Let $A=(a_{ij})\in \mathbb C^{n\times n}$ be an $SDD_{1}$ matrix with $|N_{2}|\geq 2$. Then
	\begin{align}
		\left\|A^{-1}\right\|_{\infty} \leq  \left(1+\underset{i\in N_{2}}{\rm max}\frac{P_{i}(A)}{|a_{ii}|}\right){\rm max}\{\phi,\psi\},\label{Formula 6.1}
	\end{align}
	where 
	\begin{align*}
		\phi= \underset{\underset{i\neq j}{i,j\in N_{2}}}{\rm max}\frac{|a_{jj}|+R^{N_{2}}_{i}(A)}{|a_{ii}||a_{jj}|-R^{N_{2}}_{i}(A)R^{N_{2}}_{j}(A)},
	\end{align*}
	and
\begin{align*}
	\psi= \underset{i\in N_{1}}{\rm max}\frac{1+\phi R^{N_{2}}_{i}(A)}{|a_{ii}|-R^{N_{1}}_{i}(A)-\underset{j\in N_{2}}\sum \frac{|a_{ij}|}{|a_{jj}|}P_{j}(A)}.
\end{align*}
\end{theorem}	
\begin{proof}
	Take $\alpha=N_{2}$.  Then there exists a permutation matrix $P$ such that
	\begin{align*}
		P^{T}AP=\left(
		\begin{array}{cc}
			A(\alpha)& A(\alpha, \overline \alpha)  \\
			A(\overline \alpha, \alpha) &A(\overline \alpha)   \\
		\end{array}
		\right).
	\end{align*}
	Furthermore, we obtain that
	\begin{align*}
		XP^{T}APY=\left(
		\begin{array}{cc}
			A(\alpha)&O  \\
			O &A/\alpha   \\
		\end{array}
		\right),
	\end{align*}
	where $X$ and $Y$ define as (\ref{XY4.1}). 	Therefore, we obtain that
	\begin{align*}
		A^{-1}=PY\left(
		\begin{array}{cc}
			A(\alpha)^{-1}&O  \\
			O &(A/\alpha)^{-1}   \\
		\end{array}
		\right)XP^{T}.
	\end{align*}
		Since the permutation matrix $P$ satisfies $\left\|P\right\|_{\infty}=1$. By (\ref{ZJFS}), it follows that
		\begin{align*}
		\left\|A^{-1}\right\|_{\infty}\leq 	\left\|Y\right\|_{\infty} {\rm max}\{	\left\|A(\alpha)^{-1}\right\|_{\infty},  \left\|(A/\alpha)^{-1} \bigl( -A(\overline{\alpha},\alpha) A(\alpha)^{-1}\thinspace I_{n-|\alpha|}\bigr) \right\|_{\infty}\}.
	\end{align*}
	Since both  $A(\alpha)$ and $A/\alpha$ are SDD matrices. From Lemma \ref{Lemma 5.1}, we have
\begin{align}
	\left\|A(\alpha)^{-1}\right\|_{\infty}\leq  	 \underset{\underset{i\neq j}{i,j\in N_{2}}}{\rm max}\frac{|a_{jj}|+R^{N_{2}}_{i}(A)}{|a_{ii}||a_{jj}|-R^{N_{2}}_{i}(A)R^{N_{2}}_{j}(A)}=\phi. \label{T61}
\end{align}	
From the fundamental property of infinity norm,  we have
	\begin{align}
		\Biggl\| (A/\alpha)^{-1} \Bigl( -A(\overline{\alpha},\alpha) A(\alpha)^{-1} \thinspace I_{n-|\alpha|} \Bigr) \Biggr\|_{\infty} 
		&\leq \Biggl\| < (A/\alpha) >^{-1} \Bigl( |A(\overline{\alpha},\alpha)| < A(\alpha) >^{-1} \thinspace I_{n-|\alpha|} \Bigr) \Biggr\|_{\infty}      \notag \\
		&\leq \max_{j_t \in \overline{\alpha}} \frac{1 + \phi R^{\alpha}_{j_t}(A)}{ |a'_{tt}| - R_t(A/\alpha) }\quad    (\text {\rm By Lemma \ref{Lemma 5.2})}\notag \\
		&\leq \max_{i \in N_1} \frac{1 + \phi R^{N_2}_i(A)}{ \displaystyle |a_{ii}| - R^{N_1}_i(A) - \!\!\! \sum_{j \in N_2} \! \frac{|a_{ij}|}{|a_{jj}|} P_j(A) } \quad   (\text {\rm By Theorem\ref{Theorem 3.2})}\notag\\
		&=\psi.\label{T62}
	\end{align}
		Note that\vspace{-1.5ex}
	\begin{align*}
		\left\|Y^{-1}\right\|_{\infty}=1+	\left\|A(\alpha)^{-1}A(\alpha, \overline \alpha)\right\|_{\infty}.
	\end{align*}
	Let $e=(1,1,\dots,1)^{T}\in \mathbb C^{n-|\alpha|}$, we have 
	\begin{align*}
		\left\|A(\alpha)^{-1}A(\alpha, \overline \alpha)\right\|_{\infty}&=	\left\||A(\alpha)^{-1}A(\alpha, \overline \alpha)|\right\|_{\infty}\\
		&\leq  \left\|<A(\alpha)>^{-1}|A(\alpha, \overline \alpha)|\right\|_{\infty}\\
		&=\left\|<A(\alpha)>^{-1}|A(\alpha, \overline \alpha)|e\right\|_{\infty}\\
		&=\left\|<A(\alpha)>^{-1} (R^{\overline \alpha}_{i_{1}}(A),\dots,R^{\overline \alpha}_{i_{k}}(A))^{T}       \right\|_{\infty}\\
		&\leq \left\|(\frac{P_{i_{1}}(A)}{|a_{i_{1}i_{1}}|},\dots,\frac{P_{i_{k}}(A)}{|a_{i_{k}i_{k}}|} )^{T}      \right\|_{\infty}=\underset{i\in N_{2}}{\rm max}\frac{P_{i}(A)}{|a_{ii}|} (\text{By Lemma \ref{Lemma 3.1}}).
	\end{align*}
	Therefore $\left\|Y^{-1}\right\|_{\infty}\leq 1+\underset{i\in N_{2}}{\rm max}\frac{P_{i}(A)}{|a_{ii}|}$. Furthermore, combine (\ref{T61}) and (\ref{T62}), we have (\ref{Formula 6.1}).
	 The proof is completed.	
\end{proof}
\begin{remark}\rm
When $|N_{2}|=1$, without loss of generality, we can assume that $N_{2}=\{i_{0}\}$. We only need to modify  $\phi$ to be $|a_{i_{0}i_{0}}|^{-1}$, and Theorem \ref{Theorem 6} still holds.
\end{remark}

\begin{example}\rm
	Consider the $SDD_{1}$ matrix
\begin{align*}
	A = \left(
	\renewcommand{\arraystretch}{1.2} 
	\setlength{\arraycolsep}{1.2em}  
	\begin{array}{@{\hspace{0.2em}} *{8}{c} @{\hspace{0.2em}} } 
		0.4506 & 0.0901 & 0.0451 & 0.0472 & 0.0912 & 0.0901 & 0.0421 & 0.1352 \\
		0.0901 & 0.5407 & 0.1352 & 0.0451 & 0.0901 & 0.4506 & 0.0471 & 0.1357 \\
		0.1352 & 0.0901 & 0.5407 & 0.0351 & 0.0611 & 0.0901 & 0.0451 & 0.0901 \\
		0.0901 & 0.0451 & 0.0451 & 0.5407 & 0.0901 & 0.2704 & 0.0428 & 0.0701 \\
		0.0261 & 0.0791 & 0.0451 & 0.0901 & 9.0118 & 0.0451 & 0.1352 & 0.0151 \\
		0.3154 & 0.0161 & 0.0901 & 0.1352 & 0.2704 & 27.0354 & 0.0451 & 0.0901 \\
		0.0142 & 0.0242 & 0.0451 & 0.1352 & 0.2704 & 0.5407 & 1.4419 & 0.2704 \\
		0.0151 & 0.0901 & 0.0451 & 0.0901 & 0.0901 & 0.0361 & 0.0241 & 9.0118
	\end{array}
	\right).
\end{align*}
We have $N_{1}=\{1,2,3,4\}$, $N_{2}=\{5,6,7,8\}$. By Theorem \ref{Theorem 6}, we obtain that
\begin{align*}
		\left\|A^{-1}\right\|_{\infty}\leq  7.6167.
\end{align*}
By Lemma \ref{Lemma 5.3}, we have
\begin{align*}
	\left\|A^{-1}\right\|_{\infty}\leq  11.1572(\epsilon=0.2122),\qquad \epsilon\in (0,0.2787).
\end{align*}
The exact value $	\left\|A^{-1}\right\|_{\infty}=3.3042$.
\end{example}

To enable a more profound exploration of the error bound estimation for the linear complementarity problem in Section \ref{section.5}, we introduce here the concept of the  $S$-$SDD_{1}$  matrix. In fact, this concept provides a more  precise description of the $SDD_{1}$ matrix.

\begin{definition} \rm \label{Definition 5.1}
	 Let  $S$ be a nonempty subset of $N_{2}$. Then  $A=(a_{ij})\in \mathbb C^{n\times n}$$(n\geq 2)$ is said to be an $S$-$SDD_{1}$ matrix, if for all $i\in N$,
	 \begin{align*}
	 	|a_{ii}|-R^{\overline S}_{i}(A)-Q^{S}_{i}(A)>0.
	 \end{align*}
\end{definition}

\begin{theorem}\rm  \label{Theorem 5.2}
	Let   $A=(a_{ij})\in \mathbb C^{n\times n}$$(n\geq 2)$ be an $S$-$SDD_{1}$ matrix, then\\
	(i) $A$ is an $H$-matrix.\\
	(ii) $A/S$ is an SDD matrix.\\
	(iii) If   $2\leq |S|\leq |N_{2}|$, we have
	\begin{align*}
		\left\|A^{-1}\right\|_{\infty} \leq  \left(1+\underset{i\in S}{\rm max}\frac{R^{\overline S}_{i}(A)+Q^{S}_{i}(A)}{|a_{ii}|}\right){\rm max}\{\phi_{S},\psi_{S}\},
	\end{align*}
	where
	\begin{align*}
		\phi_{S}= \underset{\underset{i\neq j}{i,j\in S}}{\rm max}\frac{|a_{jj}|+R^{S}_{i}(A)}{|a_{ii}||a_{jj}|-R^{S}_{i}(A)R^{S}_{j}(A)},
	\end{align*}	
	and
	\begin{align*}
		\psi_{S}= \underset{i\in\overline S}{\rm max}\frac{1+\phi_{S}R^{S}_{i}(A)}{|a_{ii}|-R^{\overline S}_{i}(A)-\underset{j\in S}\sum \frac{|a_{ij}|}{|a_{jj}|}\left(R^{\overline S}_{j}(A)+Q^{S}_{j}(A)\right)}.
	\end{align*}
\end{theorem}
\begin{proof}
	(i) Note that $\emptyset \neq S\subseteq N_{2} $. For all $i\in N$, we have
	\begin{align*}
		|a_{ii}|-P_{i}(A)\geq 	|a_{ii}|-R^{\overline S}_{i}(A)-Q^{S}_{i}(A)>0.
	\end{align*}
Then $A$ is an	$SDD_{1}$ matrix, therefore $A$ is an $H$-matrix.\\
(ii)    Note that
\begin{align*}
	|a^{\prime}_{tt}|-R_{t}(A/S)&\geq |a_{j_{t}j_{t}}|-R^{\overline S}_{j_{t}}(A)-\underset{j_{u}\in \overline S}\sum\Delta_{j_{t}j_{u}}.
\end{align*}
	Owing to Lemma \ref{Lemma 3.1}, we have
	\begin{align*}
		\underset{j_{u}\in \overline S}\sum\Delta_{j_{t}j_{u}}&=|a_{j_{t}S}|<A(S)>^{-1}\left(R^{\overline S}_{i_{1}}(A),\dots, R^{\overline S}_{i_{k}}(A)\right)^{T}
		\leq  \underset{h\in S}\sum \frac{|a_{j_{t}h}|}{|a_{hh}|}\left( R^{\overline S}_{h}(A)+Q^{S}_{h}(A)           \right).
	\end{align*}
	Therefore, for any $j_{t}\in \overline S$ we have
	\begin{align*}
		|a^{\prime}_{tt}|-R_{t}(A/S)\geq& |a_{j_{t}j_{t}}|-R^{\overline S}_{j_{t}}(A)- \underset{h\in S}\sum \frac{|a_{j_{t}h}|}{|a_{hh}|}\left( R^{\overline S}_{h}(A)+Q^{S}_{h}(A)           \right)\\
		\geq& 	|a_{j_{t}j_{t}}|-R^{\overline S}_{j_{t}}(A)-Q^{S}_{j_{t}}(A)>0.
	\end{align*}
	Then $A/S$ is an SDD matrix. The properties given by (ii), similar to the proof of Theorem  \ref{Theorem 6}, we can prove (iii). The proof is completed.
\end{proof}
\begin{remark}\rm  
	Recalling the definition of (\ref{eq 2.1}), we observe that $SDD_{1}$ matrices compensate all non-diagonally dominant rows by utilizing all diagonally dominant rows to achieve weakened diagonal dominance. In contrast,   $S$-$SDD_{1}$ matrices compensate non-diagonally dominant rows using only a subset of diagonally dominant rows.
	From this perspective, the $S$-$SDD_{1}$	matrices are closer to SDD matrices than $SDD_{1}$.
	\end{remark}

\section{The error bound for  linear complementarity problems with a  $B_{1}$-matrix }\label{section.5}
In this section, we derive a  bound for linear complementarity problems of $B_{1}$-matrices.   Next, let us review the specific definition of the $B_{1}$-matrix:
\begin{definition}\rm \cite{Pena} Let  $M=(m_{ij})\in{\mathbb R^{n\times n}}$ be written in the form 
	\begin{align}
		M=A+C,  \label{eq 6.2}
	\end{align}
	where
	\begin{align*}
		A=(a_{ij})=\left(
		\begin{array}{ccc}
			m_{11}-r_{1}& \dots &m_{1n}-r_{1} \\
			\vdots      &       &\vdots \\
			m_{n1}-r_{n}&\dots  & m_{nn}-r_{n} 
		\end{array}
		\right)  
		,\quad  C=\left(
		\begin{array}{ccc}
			r_{1}&\dots &r_{1}\\
			\vdots&      &\vdots \\
			r_{n}&\dots       &r_{n}    
		\end{array}
		\right),
	\end{align*}
	and
	\begin{align*}
		r_{i}:={\rm max}\{0, {\rm \underset{i\neq j}{max}}\thinspace m_{ij}\}. 
	\end{align*}
	We call $M$ a $B_{1}$-matrix if $A$ is an  $SDD_{1}$ matrix, and  all its diagonal entries are positive.
\end{definition}

\begin{lem}\rm \cite{Pena}
	If $M$ is a $B_{1}$-matrix then $M$ is a $P$-matrix.
\end{lem}

Subsequently, we establish an upper bound for (\ref{0.1})  when $M$ is a $B_{1}$-matrix.  Through the substitution $M=A+C$ into $I-D+DM$, we obtain the decomposition:
\begin{align}
	I-D+DM=B+T, \label{eq 6.3} 
	\end{align}
where $B=I-D+DA$ and $T=DC$. The properties of $B$ are closely related to the final upper bound. We give an explanation of its properties by the following Theorem:

\begin{theorem}\rm \label{Theorem 6.1}
	Let $A=(a_{ij})\in \mathbb C^{n\times n}$ be an $SDD_{1}$ matrix with all positive diagonal entries. And let $B=I-D+DA=(b_{ij})$, where $D=diag\{d_{1},d_{2},\dots, d_{n}\}$ with $0\leq d_{i}\leq 1$. Then $B$ is an $SDD_{1}$ matrix with diagonal entries are all positive. Denote $N_{1}(B)$ as the non-diagonally dominant set of $B$ and $N_{2}(B)$ as the diagonally dominant set of $B$, the containment relations hold:
	\begin{align}
		N_{1}(B)\subseteq N_{1}, N_{2}\subseteq N_{2}(B). \label{eq 6.4}
	\end{align}
		Furthermore,  $B$ is an $N_{2}$-$SDD_{1}$ matrix.
\end{theorem}
\begin{proof}
	For any $i\in N$, due to (ii) of Lemma \ref{Lemma 2.6}, we obtain  that
	\begin{align*}
		b_{ii}=1-d_{i}+d_{i}a_{ii}\geq {\rm min} \{1,a_{ii}\}>0.
	\end{align*}
More carefully, 	for any $i\in N_{2}$, it follows that
	\begin{align*}
		b_{ii}-R_{i}(B)=1-d_{i}+d_{i}a_{ii}-d_{i}R_{i}(A)\geq {\rm min}\{1, a_{ii}-R_{i}(A)\}>0.
	\end{align*}
	This means  $N_{2}\subseteq N_{2}(B)$. Since  $\overline {N_{2}(B)}=N_{1}(B)$, we consequently have $N_{1}(B)\subseteq N_{1}$.
	For any $i\in N$, we have 
	\begin{align*}
		b_{ii}-P_{i}(B)=&b_{ii}-R^{N_{1}(B)}_{i}(B)-Q^{N_{2}(B)}_{i}(B)\\
		\geq & b_{ii}-R^{N_{1}}_{i}(B)-Q^{N_{2}}_{i}(B)\\
	    =& 1-d_{i}+d_{i}\left(a_{ii}-R^{N_{1}}_{i}(A)-\underset{j\in N_{2}\backslash \{i\}}\sum \frac{d_{j}|a_{ij}|R_{j}(A)}{1-d_{j}+d_{j}a_{jj}}\right)\\
	    \geq &1-d_{i}+d_{i}\left(a_{ii}-R^{N_{1}}_{i}(A)-Q^{N_{2}}_{i}(A)        \right) (\text{By the (i) of Lemma \ref{Lemma 2.6}})\\
	    \geq &{\rm min} \{1, a_{ii}-P_{i}(A)       \}>0.
	\end{align*}
In summary,  $B$ is an $SDD_{1}$  matrix with all diagonal entries  positive.  Furthermore, by definition \ref{Definition 5.1},  $B$ constitutes an  $N_{2}$-$SDD_{1}$ matrix.
\end{proof}
The following theorem presents a  bound for $	\underset{\rm d\in[0,1]^{n}}{\rm{max}}\left\|(I-D+DM)^{-1}\right\|_{\infty}$  with a $B_{1}$-matrix $M$.

\begin{theorem}\rm \label{Theorem 6.2} Let $M=(m_{ij})\in \mathbb R^{n\times n}$ be a $B_{1}$-matrix with the from $M=A+C$, where $A$ and $C$ are given as in $(\ref{eq 6.2})$ and $|N_{2}|\geq 2$. Then
	\begin{align}
			\underset{\rm d\in[0,1]^{n}}{\rm{max}}\left\|(I-D+DM)^{-1}\right\|_{\infty}\leq (n-1)(1+\underset{i\in N_{2}}{\rm max}\frac{P_{i}(A)}{a_{ii}}){\rm max}\{\widetilde \phi, \widetilde \psi     \},\label{Formula 9.1}
	\end{align}
	where
	\begin{align*}
		\widetilde \phi= \underset{\underset{i\neq j}{i,j\in N_{2}}}{\rm max}\frac{{\rm max} \{1,a_{jj}\}+R^{N_{2}}_{i}(A)      }{{\rm min} \{1,a_{ii},a_{jj},a_{ii}a_{jj}-R^{N_{2}}_{i}(A)R^{N_{2}}_{j}(A)\}         },
	\end{align*}
	and
	\begin{align*}
		\widetilde \psi =\underset{i\in N_{1}     }{ \rm max   }\frac{1+\widetilde \phi R^{N_{2}}_{i}(A)}{{\rm min}\{1,a_{ii}-R^{N_{1}}_{i}(A)-\underset{j\in N_{2}}\sum\frac{|a_{ij}|}{a_{jj}}P_{j}(A)   \} }.
	\end{align*}
	 If $B$ is also a $Z$-matrix, The coefficient $n-1$ can be removed, i.e., \vspace{-1ex}
		\begin{align}
		\underset{\rm d\in[0,1]^{n}}{\rm{max}}\left\|(I-D+DM)^{-1}\right\|_{\infty}\leq (1+\underset{i\in N_{2}}{\rm max}\frac{P_{i}(A)}{a_{ii}}){\rm max}\{\widetilde \phi, \widetilde \psi     \}.\label{Formula 9.2}
	\end{align}
\end{theorem}

\begin{proof}
Since  $B$ is a $B_{1}$-matrix, we have $A$ is an $SDD_{1}$ matrix with all positive diagonal entries. Owing to decomposition (\ref{eq 6.3}), we obtain that $I-D+DM=B+T$, where $B=I-D+DA$, $T=DC$.  Similarly to the proof of Theorem 2 in \cite{Gar} , we have  
 \begin{align*}
	\left\|(I-D+DM)^{-1}\right\|_{\infty}= \left\|(I+B^{-1}T)^{-1}B^{-1}\right\|_{\infty}\leq \left\|(I+B^{-1}T)^{-1}\right\|_{\infty}\left\|B^{-1}\right\|_{\infty}\leq (n-1) \left\|B^{-1}\right\|_{\infty}.
\end{align*}
Owing to Theorem \ref{Theorem 6.1}, we have $B$ is an $N_{2}$-$SDD_{1}$ matrix with all diagonal entries are positive. Using Theorem \ref{Theorem 5.2}, we obtain that 
	\begin{align}
	\left\|B^{-1}\right\|_{\infty} \leq  \left(1+\underset{i\in N_{2}}{\rm max}\frac{\hat{P}_{i}(B)}{|b_{ii}|}\right){\rm max}\{\hat{\phi},\hat{\psi}\},\label{eq 6.5}
\end{align}
where 
\begin{align*}
	\hat{\phi}= \underset{\underset{i\neq j}{i,j\in N_{2}}}{\rm max}\frac{|b_{jj}|+R^{N_{2}}_{i}(B)}{|b_{ii}||b_{jj}|-R^{N_{2}}_{i}(B)R^{N_{2}}_{j}(B)},
\end{align*}
and 
\begin{align*}
	\hat{\psi}= \underset{i\in N_{1}}{\rm max}\frac{1+\hat\phi R^{N_{2}}_{i}(B)}{|b_{ii}|-R^{N_{1}}_{i}(B)-\underset{j\in N_{2}}\sum \frac{|b_{ij}|}{|b_{jj}|}\hat{P}_{j}(B)},
\end{align*}
and 
\begin{align*}
	\hat{P}_{i}(B)=&R^{N_{1}}_{i}(B)+Q^{N_{2}}_{i}(B)\\
	=&d_{i}\left( R^{N_{1}}_{i}(A)+\underset{j\in N_{2}\backslash \{i\}}\sum \frac{d_{j}|a_{ij}|R_{j}(A)        }{1-d_{j}+d_{j}a_{jj}}        \right)\\
	\leq&d_{i} \left( R^{N_{1}}_{i}(A)+Q^{N_{2}}_{i}(A) \right) \\
	=&d_{i}P_{i}(A).
\end{align*}
 Consequently, it follows that 
 \begin{align*}
 		1+\underset{i\in N_{2}}{\rm max}\frac{\hat{P}_{i}(B)}{|b_{ii}|}\leq 	1+\underset{i\in N_{2}}{\rm max}\frac{d_{i}{P}_{i}(A)}{1-d_{i}+d_{i}a_{ii}}\leq 	1+\underset{i\in N_{2}}{\rm max}\frac{P_{i}(A)}{a_{ii}}.
 \end{align*}
 We now derive an estimate for  $\hat \phi$ using the elements of $A$. Denote 
\begin{align}
	\Gamma_{ij}=b_{ii}b_{jj}-R^{N_{2}}_{i}(B)R^{N_{2}}_{j}(B).\label{eq 6.6}
\end{align}
Then  \vspace{-2ex}
\begin{align*}
	\frac{\Gamma_{ij}}{b_{jj}}=&b_{ii}-R^{N_{2}}_{i}(B)\frac{R^{N_{2}}_{j}(B)}{b_{jj}}\\
	=&1-d_{i}+d_{i}\left(a_{ii}-R^{N_{2}}_{i}(A)\frac{R^{N_{2}}_{j}(B)}{b_{jj}}\right)\\
	\geq & {\rm min}\{1, a_{ii}-R^{N_{2}}_{i}(A)\frac{R^{N_{2}}_{j}(B)}{b_{jj}}        \}.
\end{align*}
If $a_{ii}-R^{N_{2}}_{i}(A)\frac{R^{N_{2}}_{j}(B)}{b_{jj}} <1$, we have 
\begin{align*}
	\frac{\Gamma_{ij}}{b_{jj}} \geq  a_{ii}-R^{N_{2}}_{i}(A)\frac{R^{N_{2}}_{j}(B)}{b_{jj}}.
\end{align*}
The above formula is equivalent to
\begin{align*}
	\Gamma_{ij} \geq  a_{ii}b_{jj}-R^{N_{2}}_{i}(A)R^{N_{2}}_{j}(B).
\end{align*}
Then \vspace{-1ex}
\begin{align*}
		\frac{\Gamma_{ij}}{a_{ii}} \geq&  b_{jj}-\frac{R^{N_{2}}_{i}(A)}{a_{ii}}R^{N_{2}}_{j}(B)\\
		=& 1-d_{j}+d_{j}a_{jj}-d_{j}\frac{R^{N_{2}}_{i}(A)}{a_{ii}}R^{N_{2}}_{j}(A)\\
		\geq& {\rm min} \{1, a_{jj}-\frac{R^{N_{2}}_{i}(A)}{a_{ii}}R^{N_{2}}_{j}(A)\}.
\end{align*}
Therefore  
\begin{align}
	\Gamma_{ij}\geq {\rm min}\{a_{ii},a_{ii}a_{jj}-R^{N_{2}}_{i}(A)R^{N_{2}}_{j}(A)\}.\label{eq 6.7}
\end{align}
If $a_{ii}-R^{N_{2}}_{i}(A)\frac{R^{N_{2}}_{j}(B)}{b_{jj}} \geq 1$, we have
\begin{align}
	\Gamma_{ij}\geq b_{jj}=1-d_{j}+d_{j}a_{jj}\geq {\rm min}\{1,a_{jj}\}.\label{eq 6.8}
\end{align}
Due to (\ref{eq 6.7}) and (\ref{eq 6.8}), it follows that
\begin{align}
	\Gamma_{ij}\geq {\rm min}\{1,a_{ii},a_{jj},a_{ii}a_{jj}-R^{N_{2}}_{i}(A)R^{N_{2}}_{j}(A)\}. \label{eq 6.9}
\end{align}
By (ii) of Lemma \ref{Lemma 2.6},  we have
\begin{align}
	|b_{jj}|+R^{N_{2}}_{i}(B)=1-d_{j}+d_{j}a_{jj}+d_{i}R^{N_{2}}_{i}(A)\leq {\rm max}\{1,a_{jj}\}+R^{N_{2}}_{i}(A). \label{eq 6.10}
\end{align}
Owing to (\ref{eq 6.9}) and (\ref{eq 6.10}), it follows that
\begin{align}
	\hat\phi=\underset{\underset{i\neq j}{i,j\in N_{2}}}{\rm max}\frac{|b_{jj}|+R^{N_{2}}_{i}(B)}{\Gamma_{ij}}\leq \underset{\underset{i\neq j}{i,j\in N_{2}}}{\rm max}\frac{{\rm max} \{1,a_{jj}\}+R^{N_{2}}_{i}(A)      }{{\rm min} \{1,a_{ii},a_{jj},a_{ii}a_{jj}-R^{N_{2}}_{i}(A)R^{N_{2}}_{j}(A)\}         }=:\widetilde \phi.\label{eq 6.11}
\end{align}
For $\hat\psi$, we derive  \vspace{-1ex}
\begin{align*} 
	 &b_{ii}-R^{N_{1}}_{i}(B)-\underset{j\in N_{2}}\sum \frac{|b_{ij}|}{|b_{jj}|}\hat{P}_{j}(B)\\
	=& 1-d_{i}+d_{i}a_{ii}-d_{i}R^{N_{1}}_{i}(A)-d_{i}\underset{j\in N_{2}}\sum \frac{|a_{ij}|}{1-d_{j}+d_{j}a_{jj}}\hat{P}_{j}(B)\\
	\geq&1-d_{i}+d_{i}a_{ii}-d_{i}R^{N_{1}}_{i}(A)-d_{i}\underset{j\in N_{2}}\sum \frac{|a_{ij}|d_{j}}{1-d_{j}+d_{j}a_{jj}}P_{j}(A)\\
	\geq& 
	 1-d_{i}+d_{i}a_{ii}-d_{i}R^{N_{1}}_{i}(A)-d_{i}\underset{j\in N_{2}}\sum\frac{|a_{ij}|}{a_{jj}}{P}_{j}(A)\\
	\geq& {\rm min}\{1,a_{ii}-R^{N_{1}}_{i}(A)-\underset{j\in N_{2}}\sum\frac{|a_{ij}|}{a_{jj}}P_{j}(A)\}.  
\end{align*}
Therefore
\begin{align}
	\hat\psi=\underset{i\in N_{1}     }{ \rm max   }\frac{1+\hat\phi R^{N_{2}}_{i}(B)}{|b_{ii}|-R^{N_{1}}_{i}(B)-\underset{j\in N_{2}}\sum \frac{|b_{ij}|}{|b_{jj}|}\hat P_{j}(B) }      \leq \underset{i\in N_{1}     }{ \rm max   }\frac{1+\widetilde \phi R^{N_{2}}_{i}(A)}{{\rm min}\{1,a_{ii}-R^{N_{1}}_{i}(A)-\underset{j\in N_{2}}\sum\frac{|a_{ij}|}{a_{jj}}P_{j}(A)   \} }=:\widetilde \psi.\label{eq 6.12}
\end{align}
Put (\ref{eq 6.11}) and (\ref{eq 6.12}) into (\ref{eq 6.5}) to get (\ref{Formula 9.1}). Now,  we consider the special case that $B$ is also a $Z$-matrix, i.e., $M=A$ and $C$ is a null matrix. Hence, $T=DC$ is a null matrix. Furthermore,  we have 
 \begin{align*}
 		\left\|(I-D+DM)^{-1}\right\|_{\infty}\leq 	\left\|B^{-1}\right\|_{\infty}.
 \end{align*}
Consequently, (\ref{Formula 9.2}) holds. The proof is completed.
\end{proof}
\begin{remark}\rm
	When $|N_{2}|=1$, without loss of generality, we can assume that $N_{2}=\{i_{0}\}$. We only need to modify  $\widetilde\phi$ to be ${\rm max}\{1,|a_{i_{0}i_{0}}|^{-1}\}$, and Theorem \ref{Theorem 6.2} still holds.
\end{remark}

\begin{example}\rm Consider the $B_{1}$-matrix
		\begin{align*}
		M=\left(
			\renewcommand{\arraystretch}{0.9} 
		\setlength{\arraycolsep}{0.6em}  
		\begin{array}{@{\hspace{0.2em}} *{8}{c} @{\hspace{0.2em}} } 
7.5 & -2      & -2      & -2      & -0.9 & 0      & -0.5 & -0.1 \\
-12     & 20.6 & -7      & 0      & -0.9 & -0.6 & -0.1 & 0      \\
-3.5 & -2.2 & 7      & 0      & -1.4 & 0      & 0      &-0.1 \\
-12.4 & -35     & -7      & 56     & -1.6 & 0      & -0.1 & 0      \\
-0.12 & 0      & 0      & 0      & 1.2 & -0.16 & 0      & 0      \\
0      & 0      & -0.2     & -0.12      & 0 & 1.2 & 0      & 0 \\
0      & 0      & -0.16 & 0      & -0.1 & 0      & 1.2      & 0 \\
-0.1 & 0      & -0.12 & 0      & 0 & 0      & 0      & 1.2 
		\end{array}
		\right).	
	\end{align*}
	We have $N_{1}=\{1,2,3,4\}$, $N_{2}=\{5,6,7,8\}$. By Theorem \ref{Theorem 5.2}, we obtain 
	\begin{align*}
		\underset{\rm d\in[0,1]^{n}}{\rm{max}}\left\|(I-D+DM)^{-1}\right\|_{\infty}	\leq 4.1952.
	\end{align*}
	
	Since matrix $M$ is an $H$-matrix with positive main diagonal entries and satisfies the conditions of Theorem 2.1 in \cite{Chen X}, applying this theorem yields 
		\begin{align*}
		\underset{\rm d\in[0,1]^{n}}{\rm{max}}\left\|(I-D+DM)^{-1}\right\|_{\infty}	\leq 10.7081.
	\end{align*}
\end{example}
As illustrated in Fig. \ref{Figure.5.2}, we generated $5000$ diagonal matrices $D$ using MATLAB and computed the corresponding $\left\|(I-D+DM)^{-1}\right\|_{\infty}$. That  $4.1952$ is better than $10.7081$ for max$\left\|(I-D+DM)^{-1}\right\|_{\infty}$.
\begin{figure}[h]
	\begin{minipage}{1.0 \linewidth}
		\centering
		\includegraphics[width=5in]{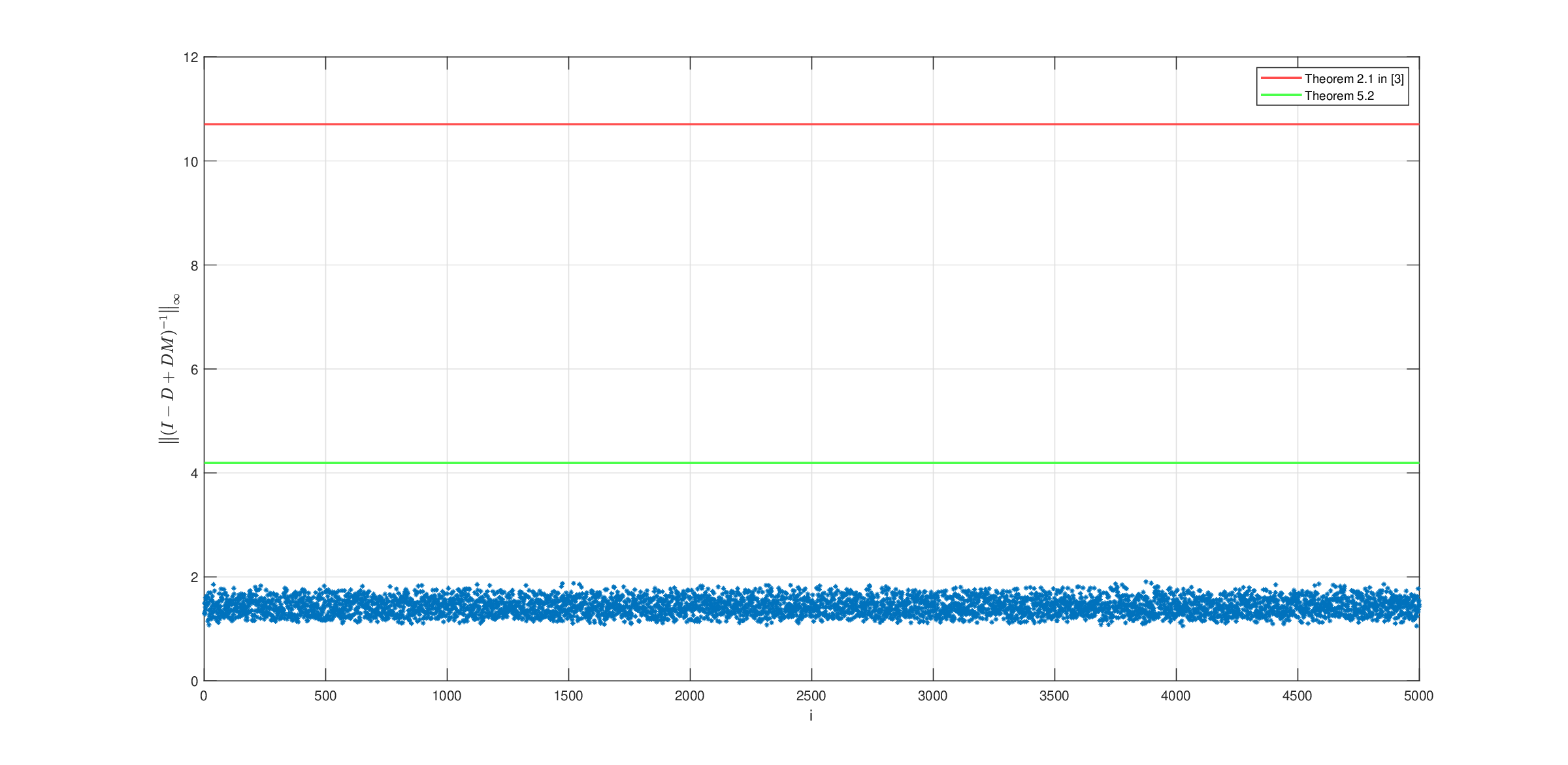}
	\end{minipage}
	\qquad	\caption{	$\left\|(I-D+DM)^{-1}\right\|_{\infty}$     for the 5000 matrices $D$ generated by diag(rand(8,1)).}\label{Figure.5.2}
\end{figure}

\section{The bounds on the determinant for $SDD_{1}$ matrices} \label{section.6}

\qquad The determinants of matrices are widely used in Accurate Computation\cite{HuangRong,Orera,Huang rong}. When the order of the matrix is too large, it is terrible to calculate the determinant directly. Therefore, scholars usually give the upper and lower bounds  of the determinant, and the common effective way is to use the elements of the original matrix to give this estimate. In this section, we will provide the   new determinant upper and lower bounds of the $SDD_{1}$ matrix by Schur complement. An interesting fact is that in 2005, before the $SDD_{1}$ matrix was formally proposed, Huang had already given the upper and lower bounds of the determinant for matrices with such a structure. We summarize them as follows:

\begin{lem}\rm  \label{Lemma 4.1}\cite{Htz} Let $A=(a_{ij})\in \mathbb C^{n\times n}$ be an $SDD_{1}$ matrix, then
	\begin{align}
		\prod\limits_{i=1}^n l_{i}\leq |\det(A)|\leq 	\prod\limits_{i=1}^nu_{i},
	\end{align}\label{eq 4.1}
	where \vspace{-2ex}
	\begin{align*}
		l_{i}=|a_{ii}|-\frac{1}{x_{i}}\sum_{j=i+1}^{n}|a_{ij}|x_{j},\quad u_{i}=|a_{ii}|+\frac{1}{x_{i}}\sum_{j=i+1}^{n}|a_{ij}|x_{j},\quad a_{n,n+1}=0,
	\end{align*}
	
	\begin{align*}
		x_{i}=&\left\{\begin{aligned}\notag
			&\qquad  1, \qquad i\in N_{1}, \\
			\thinspace
			&\theta +\frac{R_{i}(A)}{|a_{ii}|},i\in N_{2},
		\end{aligned}\right.  \quad {\text{and}} \quad \theta=\underset{j\in N_{1}}{\rm min}\frac{|a_{jj}|-P_{j}(A)}{R^{N_{2}}_{j}(A)}.
	\end{align*}
	
\end{lem}

We say that the index set $N$ of matrix $A$ admits a $D_{1}$ ordering if it is partitioned as 
\begin{align}
	N_{1}=\{1,2,\dots, s\},\qquad N_{2}=\{s+1,s+2,\dots, n\},\label{eq 4.2}
\end{align}
where $0<s<n$. Obviously, for any matrix $A$, there exists a permutation $P$ such that the index set of  $P^{T}AP$ satisfies (\ref{eq 4.2}).

\begin{theorem}\rm \label{Theorem modify1}
		Let $A=(a_{ij})\in \mathbb C^{n\times n}$ be an $SDD_{1}$ matrix, and the index set $N$ of $A$ satisfies (\ref{eq 4.2}). Then
		\begin{align}
		\prod\limits_{i=1}^n f_{i}\leq |\det(A)|\leq 	\prod\limits_{i=1}^n g_{i},\label{eq 4.3}
	\end{align}
	where \vspace{-2ex}
		\begin{align*}
		f_{i}=|a_{ii}|-\sum_{j=i+1}^{n}|a_{ij}|y_{j},\quad g_{i}=|a_{ii}|+\sum_{j=i+1}^{n}|a_{ij}|y_{j},\quad a_{n,n+1}=0,
	\end{align*}
	and
		\begin{align*}
		y_{i}=&\left\{\begin{aligned}\notag
			&  \frac{P_{i}(A)}{|a_{ii}|}, \qquad i\in \{1,2,\dots, s\}, \\
			\thinspace
			&\frac{R_{i}(A)}{|a_{ii}|},\quad i\in \{s+1,\dots, n\}.
		\end{aligned}\right.
	\end{align*}
	\end{theorem}
	\begin{proof}
			Denote $\alpha_{n}=\{1,2,\dots,n\}$, $\alpha_{n-1}=\{2,3,\dots,n\}$, \dots , $\alpha_{1}=\{n\}$, $\alpha_{0}=\emptyset$.  By applying  Lemma \ref{Lemma 2.5}, we deduce that \vspace{-1.2ex}
		\begin{align*}
			\det(A)&=\det\left(A(\alpha_{n-1})\right)\times\det\left(A/A(\alpha_{n-1})\right)\\ \notag
			&=\det\left(A(\alpha_{n-2})\right)\times\det\left(A(\alpha_{n-1})/A(\alpha_{n-2})\right)\times\det\left(A/A(\alpha_{n-1})\right)\\ \notag
			&\enspace \vdots \\ \notag
			&=\det\left(A(\alpha_{1})\right)\times\det\left(A(\alpha_{2})/A(\alpha_{1})\right)\times\dots\times\det\left(A/A(\alpha_{n-1})\right).\tag{42} \label{modify1}
		\end{align*}
		For  $i\in \{1,2,\dots, s-1\}$, we have
		\begin{align*}
			&|\det\left(A(\alpha_{n-i+1})/A(\alpha_{n-i})\right)|\\ \notag
			=&|A(\alpha_{n-i+1})/A(\alpha_{n-i})|\\  \notag
			=&|a_{ii}-a_{i\alpha_{n-i}}(A(\alpha_{n-i}))^{-1}a_{\alpha_{n-i}i}|\\  \notag
			\geq &|a_{ii}|-|a_{i\alpha_{n-i}}|<A(\alpha_{n-i})>^{-1}|a_{\alpha_{n-i}i}|\\ \notag 
			\geq &|a_{ii}|-|a_{i\alpha_{n-i}}|(\frac{P_{i+1}(A)}{|a_{i+1i+1}|},\dots, \frac{P_{n}(A)}{|a_{nn}|})^{T}(\text {\rm By (\ref{eq 3.10})})\\ \notag
			=&|a_{ii}|-\sum_{j=i+1}^{n}|a_{ij}|\frac{P_{j}(A)}{|a_{jj}|}\\ \notag
			\geq&|a_{ii}|-\sum_{j=i+1}^{s}|a_{ij}|\frac{P_{j}(A)}{|a_{jj}|}-\sum_{j=s+1}^{n}|a_{ij}|\frac{R_{j}(A)}{|a_{jj}|}. \tag{43}\label{modify2}
		\end{align*}
Conversely, 
		\begin{align*}
			&|\det\left(A(\alpha_{n-i+1})/A(\alpha_{n-i})\right)|\\ 
			\leq &|a_{ii}|+|a_{i\alpha_{n-i}}|<A(\alpha_{n-i})>^{-1}|a_{\alpha_{n-i}i}|\\
			\leq &|a_{ii}|+\sum_{j=i+1}^{n}|a_{ij}|\frac{P_{j}(A)}{|a_{jj}|}\\
			\leq &|a_{ii}|+\sum_{j=i+1}^{s}|a_{ij}|\frac{P_{j}(A)}{|a_{jj}|}+\sum_{j=s+1}^{n}|a_{ij}|\frac{R_{j}(A)}{|a_{jj}|}.\tag{44} \label{modify2.1}
			\end{align*}
 Similarly, for  $i\in \{s,s+1,\dots,n\}$, we have
		\begin{align*}
			&|\det\left(A(\alpha_{n-i+1})/A(\alpha_{n-i})\right)|\\  \notag
			\geq &|a_{ii}|-|a_{i\alpha_{n-i}}|<A(\alpha_{n-i})>^{-1}|a_{\alpha_{n-i}i}|\\ \notag
			\geq &|a_{ii}|-|a_{i\alpha_{n-i}}|(\frac{|a_{i+1i}|+Q^{\alpha_{n-i}}_{i+1}(A)}{|a_{i+1i+1}|},\dots,\frac{|a_{ni}|+Q^{\alpha_{n-i}}_{n}(A)}{|a_{nn}|})^{T}(\text {\rm By  Lemma \ref{Lemma 3.1}} )\\ \notag
			\geq &|a_{ii}|-|a_{i\alpha_{n-i}}|(\frac{R_{i+1}(A)}{|a_{i+1i+1}|},\dots, \frac{R_{n}(A)}{|a_{nn}|})^{T}\\ \notag
			=&|a_{ii}|-\sum_{j=i+1}^{n}|a_{ij}|\frac{R_{j}(A)}{|a_{jj}|}.\tag{45}\label{modify3}
		\end{align*}
	Conversely, 
			\begin{align*}
			&|\det(A(\alpha_{n-i+1})/A(\alpha_{n-i}))|\\ 
			\leq &|a_{ii}|+|a_{i\alpha_{n-i}}|<A(\alpha_{n-i})>^{-1}|a_{\alpha_{n-i}i}|\\ 
			\leq &|a_{ii}|+|a_{i\alpha_{n-i}}|(\frac{R_{i+1}(A)}{|a_{i+1i+1}|},\dots, \frac{R_{n}(A)}{|a_{nn}|})^{T}\\ 
			=&|a_{ii}|+\sum_{j=i+1}^{n}|a_{ij}|\frac{R_{j}(A)}{|a_{jj}|}.\tag{46}\label{modify3.1}
		\end{align*}		
		Combine (\ref{modify2})---(\ref{modify3.1})  to (\ref{modify1}), we have (\ref{eq 4.3}). The proof is completed.
	\end{proof}

The following theorem shows that in  $D_{1}$ ordering, the upper and lower bounds of the determinant of Theorem \ref{Theorem modify1} are superior to Lemma \ref{Lemma 4.1}.
\begin{theorem}\rm 
	Let $A=(a_{ij})\in \mathbb C^{n\times n}$ be an $SDD_{1}$ matrix, and $N$ satisfies (\ref{eq 4.2}). Then
	\begin{align*}
		\prod\limits_{i=1}^n l_{i}\leq 				\prod\limits_{i=1}^n f_{i}\leq |\det(A)|\leq 	\prod\limits_{i=1}^n g_{i} 	\leq 	\prod\limits_{i=1}^nu_{i},
	\end{align*}
	where $l_{i}, u_{i}$ are given as in Lemma \ref{Lemma 4.1}, and $f_{i}, g_{i}$ are given as in Theorem \ref{Theorem modify1}.
\end{theorem}
\begin{proof}Let us prove it by discussing $i$ in two cases:\\
	When $i\in \{1,2,\dots,s\}$, we have
	\begin{align*}
		f_{i}-l_{i}=\sum_{j=i+1}^{s}|a_{ij}|(1-\frac{P_{j}(A)}{|a_{jj}|})+\sum_{j=s+1}^{n}|a_{ij}|(\theta +\frac{R_{j}(A)}{|a_{jj}|}-\frac{R_{j}(A)}{|a_{jj}|})\geq 0. 
	\end{align*}
	When $i\in \{s+1,s+2,\dots, n\}$, we have
	\begin{align*}
		f_{i}-l_{i}=\sum_{j=i+1}^{n}|a_{ij}|(\frac{\theta +\frac{R_{j}(A)}{|a_{jj}|}}{\theta+\frac{R_{i}(A)}{|a_{ii}|} }-\frac{R_{j}(A)}{|a_{jj}|}). 
	\end{align*}	
	If $\frac{R_{i}(A)}{|a_{ii}|}> \frac{R_{j}(A)}{|a_{jj}|}$, make use of Lemma \ref{Lemma 2.4}, we obtain that 
	\begin{align*}
		\frac{\theta+ \frac{R_{j}(A)}{|a_{jj}|}}{\theta +\frac{R_{i}(A)}{|a_{ii}|}}\geq \frac{ \frac{R_{j}(A)}{|a_{jj}|}}{ \frac{R_{i}(A)}{|a_{ii}|}}\geq \frac{R_{j}(A)}{|a_{jj}|}.
	\end{align*}
	If $\frac{R_{i}(A)}{|a_{ii}|}\leq  \frac{R_{j}(A)}{|a_{jj}|}$, we have
	\begin{align*}
		\frac{\theta+ \frac{R_{j}(A)}{|a_{jj}|}}{\theta +\frac{R_{i}(A)}{|a_{ii}|}} \geq 1 >\frac{R_{j}(A)}{|a_{jj}|}.
	\end{align*}
	In combination with the above steps, we have $f_{i}-l_{i}\geq 0$. By the arbitrariness of $i$, 	$	\prod\limits_{i=1}^n l_{i}\leq 				\prod\limits_{i=1}^n f_{i}$. Similarly,  we have $	\prod\limits_{i=1}^n g_{i} 	\leq 	\prod\limits_{i=1}^nu_{i}$. The proof is completed.
\end{proof}

\begin{example}\rm Consider the $SDD_{1}$  matrix 
	\begin{align*}
		A=\left(
		\begin{array}{cccccc}
			1.5 & 0.3 & 0.3 & 0.6 & 0.3 & 0.3 \\
		0.3 & 1.5 & 0.3 & 0.3 & 0.6 & 0.3 \\ 
		0.3 & 0 & 1.5 & 0.6 & 0 & 0.6 \\
		0.3 & 0 & 0.3 & 3 & 0 & 0 \\
		0.3 & 0 & 0 & 0 & 1.5 & 0.3 \\
		0 & 0 & 0.3 & 0 & 0 & 1.5 
		\end{array}
		\right).	
		\end{align*}
		We have $N_{1}=\{1,2,3\}$, $N_{2}=\{4,5,6\}$. Applying  Lemma \ref{Lemma 4.1} yields 
		\begin{align*}
			0.099\leq |\det(A)|\leq 125.8959.
		\end{align*}
		By Theorem \ref{Theorem 6.1}, we have
			\begin{align*}
			7.5835\leq |\det(A)|\leq 50.4812.
		\end{align*}
		The exact value  $\det(A)=17.6899$.
	\end{example}

\begin{example}\rm \label{example4} Consider the $SDD_{1}$ matrix
	\begin{align*}
	A=\left(
	\begin{array}{cccccc}
3 & 0 & 1 & 2 & 0 & 2 \\
1 & 2 & 0 & 1 & 0 & 0 \\
1 & 0 & 2 & 1 & 0 & 0 \\
0 & 1 & 0 & 3 & 0 & 0 \\
0 & 0 & 0 & 0 & 3 & 1 \\
0 & 0 & 0 & 1 & 0 & 3 
	\end{array}
	\right).	
\end{align*}
		Then $N_{1}=\{1,2,3\}$ and $N_{2}=\{4,5,6\}$. We  calculate the following values:
		\begin{align*}
			&R_{1}(A) = 5, \quad R_{2}(A) = 2, \quad R_{3}(A) = 2, \\
			&R_{4}(A) = 1, \quad R_{5}(A) = 1, \quad R_{6}(A) = 1, 
		\end{align*}
and $\theta=\frac{1}{6}$. 		Therefore, we have
		\begin{align*}
			l_{1}=3-1-(\frac{1}{6}+\frac{1}{3})\times2-(\frac{1}{6}+\frac{1}{3})\times 2=0.
		\end{align*}
		Applying  Lemma \ref{Lemma 4.1} yields 
\begin{align*}
	0\leq |\det(A)|\leq 1311.3.
\end{align*}
By Theorem \ref{Theorem 6.1}, we have
\begin{align*}
66.7\leq |\det(A)|\leq816.7.
\end{align*}
The exact value $\det(A)=240$.
\end{example}

\begin{remark}\rm 
From Example \ref{example4}, we observe that when the $SDD_{1}$ matrix $A$ is subjected to more adverse conditions, the determinant lower bound produced by Lemma \ref{Lemma 4.1} may degenerate to $0$. However,  for any $i\in N$, \vspace{-0.5em}
\begin{align*}
		f_{i}=&|a_{ii}|-\sum_{j=i+1}^{n}|a_{ij}|y_{j}\geq |a_{ii}|-P_{i}(A)>0,
\end{align*}
	the lower bound derived by Theorem \ref{Theorem modify1} is guaranteed to be strictly positive. This conclusively demonstrates the superiority of Theorem \ref{Theorem modify1}  over Lemma  \ref{Lemma 4.1} .
\end{remark}

\thinspace

\noindent \textbf{Acknowledgements:} The authors are thankful to the editors and the anonymous  referees for their valuable comments to improve the paper.\\

\noindent \textbf{Funding:}  The work is supported  partly by  National Key Research and Development Program of China (No. 2023YFB3001604), and Postdoctoral Fellowship Program of CPSF (Grant No.
GZC20231534).\\

\noindent \textbf{Conflict of Interest:} All authors declare that they have no conflict of interest.

\bibliography{sn-bibliography}

\section*{}

\end{document}